\documentclass{amsart}
\usepackage{amssymb,amsbsy}
\usepackage[all]{xy}
\usepackage{stmaryrd}

\newtheorem{thm}{Theorem}[section]
\newtheorem{prop}[thm]{Proposition} 
\newtheorem{lemma}[thm]{Lemma}
\newtheorem{cor}[thm]{Corollary}

\theoremstyle{definition} 
\newtheorem{dfn}[thm]{Definition}

\theoremstyle{remark}
\newtheorem{remark}[thm]{Remark}

\newtheorem{ex}[thm]{Example}

\numberwithin{equation}{section}

\newcommand{\vbgs}[4]{
\xymatrix{#1 \ar[d] \ar@<2pt>[r] \ar@<-2pt>[r] & #2 \ar[d] \\ #3 \ar@<2pt>[r] \ar@<-2pt>[r] & #4}}

 \newcommand{\arrows}{\rightrightarrows}
 \newcommand{\defequal}{:=}
 \newcommand{\tolabel}[1]{\stackrel{#1}{\to}}
 \newcommand{\reals}{{\mathbb R}}
 \newcommand{\bitimes}[2]{\,_{#1}\!\times_{#2}}
 
 \newcommand{\VB}{\mathcal{VB}}
 \newcommand{\iso}{\cong}
 \newcommand{\toiso}{\stackrel{\sim}{\to}}
 \DeclareMathOperator{\im}{im}
 \DeclareMathOperator{\coker}{coker}
 \DeclareMathOperator{\Hom}{Hom}
 \newcommand{\boundary}{\partial}
 \newcommand{\Etot}{\mathcal{E}}
 \newcommand{\totaldiff}{\mathcal{D}}
 \newcommand{\id}{\mathrm{id}}
 \newcommand{\nue}[1]{\mathring{#1}}
 \newcommand{\lin}{\mathrm{lin}} %lin = linear
 \newcommand{\framegpd}{\mathcal{G}} %frame groupoid
 \newcommand{\framecat}{\mathcal{C}} %frame category
 \newcommand{\fat}[1]{\hat{\mathcal{G}}(#1)} %fat groupoid
 \newcommand{\fatcat}[1]{\hat{\mathcal{C}}(#1)} %fat category
 \newcommand{\orbit}{\mathcal{O}} %orbit
 \newcommand{\vbiso}{\Theta} %isomorphism from C_\VB to C(G;C+E[-1])
 \newcommand{\suchthat}{\mid}
 \newcommand{\pair}[2]{\langle #1 \; \vert \; #2 \rangle}
 \newcommand{\du}[1]{\check #1}
 \newcommand{\modsim}{\!\sim}

\begin{document}
\title{$\VB$-groupoids and representation theory of Lie groupoids}
\author{Alfonso Gracia-Saz}
\address{Department of Mathematics\\
University of Toronto\\
40 Saint George Street, Room 6290\\
Toronto, Ontario, Canada M5S 2E4}
\email{alfonso@math.toronto.edu}
\author{Rajan Amit Mehta}
\address{Department of Mathematics \& Statistics\\
Smith College\\
44 College Lane\\
Northampton, MA 01063}
\email{rmehta@smith.edu}


\begin{abstract}
	A $\VB$-groupoid is a Lie groupoid equipped with a compatible linear structure. In this paper, we describe a correspondence, up to isomorphism, between $\VB$-groupoids and $2$-term representations up to homotopy of Lie groupoids. Under this correspondence, the tangent bundle of a Lie groupoid $G$ corresponds to the ``adjoint representation'' of $G$. The value of this point of view is that the tangent bundle is canonical, whereas the adjoint representation is not.

	We define a cochain complex that is canonically associated to any $\VB$-groupoid. The cohomology of this complex is isomorphic to the groupoid cohomology with values in the corresponding representations up to homotopy. When applied to the tangent bundle of a Lie groupoid, this construction produces a canonical complex that computes the cohomology with values in the adjoint representation.

	Finally, we give a classification of regular $2$-term representations up to homotopy. By considering the adjoint representation, we find a new cohomological invariant associated to regular Lie groupoids.
\end{abstract}

\maketitle

%====
\section{Introduction}

Let $G$ be a Lie group with a representation on a vector space $V$, and let $G \ltimes V$ be the semidirect product. In addition to being a group, $G \ltimes V$ is a vector bundle over $G$, and the multiplication map $(G \ltimes V) \times (G \ltimes V) \to G \ltimes V$ is linear. In other words, the semidirect product is a \emph{group object in the category of vector bundles}.

Conversely, let $\Gamma \to G$ be a group object in the category of (smooth) vector bundles. That is, let $\Gamma \to G$ be a vector bundle equipped with a Lie group structure, such that the multiplication map $\Gamma \times \Gamma \to \Gamma$ is linear. Then $G$ automatically inherits a Lie group structure. Furthermore, using right-translation by zero-vectors, we can trivialize $\Gamma$, allowing us to canonically identify it with a semidirect product $G \ltimes \Gamma_e$, where $\Gamma_e$ is the fiber over the identity element $e \in G$.

Thus there is a one-to-one correspondence between Lie group representations and group objects in the category of vector bundles. In particular, the adjoint representation of $G$ on its Lie algebra $\mathfrak{g}$ corresponds to the tangent bundle $TG$, and the coadjoint representation corresponds to the cotangent bundle $T^*G$.

The goal of this paper is to extend the above correspondence to the setting of Lie groupoids. The situation here is complicated by the fact that Lie groupoids do not in general possess well-defined adjoint representations. Rather, as was observed by Evens, Lu, and Weinstein \cite{elw}, there is a sense in which a Lie groupoid $G \arrows M$ possesses a natural ``representation up to homotopy'' on the $2$-term complex $A \to TM$, where $A$ is the Lie algebroid of $G$. 

The notion of representation up to homotopy was refined by Arias Abad and Crainic \cite{aba-cra:rephomgpd}, who gave explicit formulas for the adjoint representation up to homotopy. However, their construction relies on the choice of a certain Ehresmann connection on $G$, so it is not canonical in the strictest sense, though it is canonical up to isomorphism.

Our approach is to consider \emph{$\VB$-groupoids} as geometric models for representations up to homotopy. Essentially, a $\VB$-groupoid is a Lie groupoid with a compatible linear structure, making it a \emph{groupoid object in the category of vector bundles}. $\VB$-groupoids were first introduced by Pradines \cite{pradines2} in relation to the theory of symplectic groupoids \cite{cdw, karasev, weinstein-symplectic} and have played an important role in the study of double structures by Kirill Mackenzie and his collaborators \cite{mac:drinfeld, mac:dblie2, mac:duality, mac:ehresmann}. Given a Lie groupoid $G \arrows M$ with Lie algebroid $A$, one can construct two naturally-associated $\VB$-groupoids: the tangent bundle $TG \arrows TM$ and the cotangent bundle $T^*G \arrows A^*$. The latter is the standard example of a symplectic groupoid.

Given a Lie groupoid $G \arrows M$ with a representation up to homotopy on a $2$-term complex $C \to E$ of vector bundles over $M$, we may construct an associated $\VB$-groupoid which can be viewed as a semidirect product of $G$ with $C \to E$. On the other hand, we will see that any $\VB$-groupoid is noncanonically isomorphic to a semidirect product. Although the resulting representation up to homotopy depends on the choice, different choices lead to representations up to homotopy that are isomorphic. We thus obtain a one-to-one correspondence, up to isomorphism, between $\VB$-groupoids and $2$-term representations up to homotopy. 

In particular, the tangent and cotangent bundles correspond to the adjoint and coadjoint representations up to homotopy. This fact is particularly pleasing, since it gives us \emph{canonical} models for the adjoint and coadjoint representations up to homotopy.

To any $\VB$-groupoid, we can associate a cochain complex that is isomorphic to the complex of groupoid cochains with values in the corresponding representation up to homotopy. As an immediate application,  we obtain a canonical model for the cohomology of a Lie groupoid with values in its adjoint representation.

The perspective of $\VB$-groupoids allows us to classify regular $2$-term representations up to homotopy. Part of the classification involves a certain cohomology class, which, in the case of the adjoint representation, becomes an invariant of the Lie groupoid itself.

This paper is the companion of an earlier paper \cite{gs-m:vbalg}, where $\VB$-algebroids were studied in relation to representations up to homotopy of Lie algebroids. Many of the results of this paper can be seen as global analogues of results in \cite{gs-m:vbalg}. It is known \cite{mac:dblie2} that $\VB$-groupoids are the global objects corresponding to $\VB$-algebroids, so the theory of $\VB$-groupoids and $\VB$-algebroids provide a natural framework for understanding differentiation and integration of representations up to homotopy (at least in the $2$-term case). In the time since a preprint version of this paper was first posted on the arXiv in 2010, this has in fact been done in \cite{bco, bch}. Particularly, in \cite{bco} the perspective of $\VB$-groupoids and $\VB$-algebroids was used to fully characterize the obstructions to integrability of $2$-term representations up to homotopy.

%--
\subsection*{Structure of the paper}

\begin{itemize}
\item  In \S\ref{sec:repres}, we recall Arias Abad and Crainic's notion of representation up to homotopy of a Lie groupoid \cite{aba-cra:rephomgpd}. 
\item  In \S\ref{sec:vbg}, we recall the definition and basic facts about $\VB$-groupoids, including the notion of horizontal lift. In \S\ref{sec:formulas}, we give formulas (depending on the choice of a horizontal lift) for the representations up to homotopy arising from a $\VB$-groupoid. The semidirect product construction appears in Example \ref{ex:semi2}.
\item In \S\ref{sec:dual}, we briefly review the construction of the dual of a $\VB$-groupoid.
\item The heart of the paper is \S\ref{sec:vbsuperrep}, where we construct a canonical cochain complex associated to any $\VB$-groupoid. A choice of horizontal lift allows us to produce a representation up to homotopy. In Appendix \ref{appendix:derivation}, we show that the representations up to homotopy arising in this manner agree with the formulas given in \S\ref{sec:formulas}. In \S\ref{sec:dependence}, we study how the representation up to homotopy depends on the choice of horizontal lift. 
\item  In \S\ref{sec:moduli}, we prove (Corollary \ref{cor:big}) that isomorphism classes of $\VB$-groupoids are in one-to-one correspondence with isomorphism classes of $2$-term representations up to homotopy. Then, in Theorem \ref{thm:classification}, we give a classification of $\VB$-groupoids (and hence $2$-term representations up to homotopy) satisfying a regularity condition.
\item  The classification result of Theorem \ref{thm:classification} involves a cohomological invariant. This characteristic class is a ``higher categorical'' invariant, in the sense that it contains information about the ``homotopy,'' rather than the ``representation.'' In \S \ref{sec:adjoint}, we consider the geometric interpretation of this invariant in the case of the $\VB$-groupoid $TG$. 
\item In \S\ref{sec:fat}, we describe a different (but equivalent) approach to constructing representations up to homotopy from $\VB$-groupoids, via something we call the \emph{fat category}. This approach helps to clarify the relationship between the idea originally outlined by Evens, Lu, and Weinstein \cite{elw} and the Arias Abad-Crainic definition of representation up to homotopy.
\end{itemize}

%--
\subsection*{Acknowledgements}

We would like to thank the Department of Mathematics at Washington University in Saint Louis, as well as Eckhard Meinrenken, for funding visits during which this work was developed.  We thank Jim Stasheff for providing useful comments on a draft of this paper.  We thank Kirill Mackenzie, Ping Xu, and Mathieu Sti\'enon for interesting discussions related to this work. We also thank the anonymous referees for their careful reading and numerous suggestions that improved the clarity of the paper.

%====
\section{Groupoid representations and representations up to homotopy}\label{sec:repres}

In this section, we review Lie groupoid representations from the cohomological point of view, from which the generalization to representations up to homotopy is straightforward. The material on representations and cohomology is standard and can be found in, e.g., \cite{mac:book}. The material on representations up to homotopy essentially follows that of \cite{aba-cra:rephomgpd}.  

The main concern of this paper is $2$-term representations up to homotopy, and in \S\S\ref{sec:2term}--\ref{sec:gauge} we specialize to this case.

We remark that, although we will be working in the smooth category, the general theory of representations up to homotopy and $\VB$-groupoids goes through in the topological category. The key points where smoothness is used are to define the ``adjoint representation'' via the tangent bundle and to prove existence of decompositions in \S\ref{sec:cores}.

\subsection{Lie groupoid representations}
Let $E \to M$ be a vector bundle. The \emph{frame groupoid} $\framegpd(E)$ is the groupoid whose set of objects is $M$ and whose morphisms are isomorphisms $E_x \toiso E_y$ for $x,y \in M$. The frame groupoid is a Lie groupoid; we refer the reader to \cite{mac:book} for details.

Let $G \arrows M$ be a Lie groupoid. A \emph{representation} of $G$ is a vector bundle $E \to M$ and a Lie groupoid morphism $\Delta: G \to \framegpd(E)$.

\begin{ex}
When $M$ is a point, then $G$ is a Lie group, $E$ is a vector space, and $\framegpd(E)$ is the general linear group on $E$. Thus we recover the usual notion of Lie group representation.
\end{ex}

\begin{ex}\label{example:pair}
When $G = M \times M$ is a pair groupoid, a representation of $G$ on $E$ is equivalent to a trivialization of $E$. When $G$ is the fundamental groupoid of a manifold $M$, then a representation of $G$ on $E$ is equivalent to a flat connection on $E$. These examples demonstrate that the notion of Lie groupoid representation is too restrictive. For example, if $E \to M$ is nontrivializable, then there do not exist any representations at all of the pair groupoid $M \times M$ on $E$.
\end{ex}

\subsection{Lie groupoid cohomology}

In order to arrive at a natural definition of representation up to homotopy, we will need to restate the definition of Lie groupoid representation in cohomological terms. We first recall the notion of Lie groupoid cohomology.

Let $G \arrows M$ be a Lie groupoid with source and target maps $s,t: G \to M$. Let $G^{(0)} \defequal M$, and for $p>0$ let $G^{(p)}$ be the manifold consisting of composable $p$-tuplets of elements of $G$. In other words, 
\[ G^{(p)} \defequal G \bitimes{s}{t} \cdots \bitimes{s}{t} G = \{(g_1, \dots, g_p) \suchthat s(g_i) = t(g_{i+1})\}. \]

The space of ($\reals$-valued) smooth groupoid $p$-cochains is $C^p(G) \defequal C^\infty(G^{(p)})$. 
There is a coboundary operator $\delta: C^p(G) \to C^{p+1}(G)$ on the space of cochains, which for $p=0$ is given by 
\begin{equation*}
	(\delta f)(g) = f(s(g)) - f(t(g))
\end{equation*}
for $f \in C^0(G) = C^\infty(M)$ and $g \in G$, and for $p>0$ given by
\begin{equation*}
\begin{split}
  	 (\delta f) (g_0, \dots, g_p) = & f(g_1, \dots, g_p) + \sum_{i=1}^p (-1)^i f(g_0, \dots, g_{i-1} g_i, \dots, g_p) \\
&+ (-1)^{p+1} f(g_0, \dots, g_{p-1})  
\end{split}
\end{equation*}
for $f \in C^p(G)$ and $(g_0, \dots, g_p) \in G^{(p+1)}$. The equation $\delta^2 = 0$ is a consequence of the groupoid axioms. The cohomology of the complex $(C^\bullet(G), \delta)$ is known as the \emph{smooth groupoid cohomology} of $G$.

There is a product $C^p(G) \times C^q(G) \to C^{p+q}(G)$, $(f_1,f_2) \mapsto f_1 \star f_2$, given by
\[ (f_1 \star f_2)(g_1, \dots, g_{p+q}) = f_1(g_1, \dots, g_p) f_2 (g_{p+1}, \dots, g_{p+q})\]
for $f_1 \in C^p(G)$, $f_2 \in C^q(G)$, and $p,q > 0$. If $p=0$, $q>0$, then
\begin{equation}\label{eqn:star1}
 (f_1 \star f_2)(g_1, \dots, g_q) = f_1(t(g_1)) f_2 (g_1, \dots, g_q),	
\end{equation}
and if $q=0$, $p>0$, then
\begin{equation}\label{eqn:star2}
(f_1 \star f_2)(g_1, \dots, g_p) = f_1(g_1, \dots, g_p) f_2 (s(g_p)).
\end{equation}
If $p=q=0$, then
\begin{equation}\label{eqn:star3}
f_1 \star f_2 = f_1 f_2.	
\end{equation}
The coboundary operator $\delta$ is a graded derivation with respect to the product:
\begin{equation*}
\delta(f_1 \star f_2) = (\delta f_1) \star f_2 + (-1)^{|f_1|} f_1 \star (\delta f_2).
\end{equation*}

A cochain $f \in C^p(G)$, $p > 0$, is called \emph{normalized} if $f$ vanishes whenever at least one of its arguments is a unit. By definition, every $0$-cochain is considered to be normalized. The space of normalized cochains is closed under the coboundary operator $\delta$ and under the product $\star$. 

\subsection{Lie groupoid cohomology with values in a representation}\label{sec:cohomologyrep}

Let $G \arrows M$ be a Lie groupoid, and let $E \to M$ be a vector bundle.  The space of smooth groupoid $p$-cochains with values in $E$ is $C^p(G;E) \defequal \Gamma ((\pi_0^p)^* E)$, where $\pi^p_0: G^{(p)} \to M$ is the identity for $p=0$ and is given by $\pi^p_0(g_1, \dots, g_p) = t(g_1)$ for $p>0$. More concretely, if $\omega \in C^p(G;E)$, then $\omega(g_1, \dots, g_p)$ is an element of $E_{t(g_1)}$.

There is a right $C(G)$-module structure on $C(G;E)$, given by 
\[ (\omega \star f)(g_1, \dots, g_{p+q}) = \omega(g_1, \dots, g_p) f(g_{p+1}, \dots, g_{p+q})\]
for $\omega \in C^p(G;E)$, $f \in C^q(G)$, and $p,q>0$. When $p$ or $q$ is zero, the formula for $\omega \star f$ is similar to equations \eqref{eqn:star1}--\eqref{eqn:star3}.

The space $C^p(G;E)$, as the space of sections of a pullback bundle, can be identified with $\Gamma(E) \otimes_{C^\infty(M)} C^p(G)$, where the tensor structure is given by $\varepsilon \phi \otimes f = \varepsilon \otimes (\phi \star f)$ for $\varepsilon \in \Gamma(E)$, $\phi \in C^\infty(M)$, and $f \in C^p(G)$. In particular, we see that $C(G;E)$ is generated as a right $C(G)$-module by $\Gamma(E) = C^0(G;E)$.

Given a representation $\Delta$ of $G$ on $E$, we can construct a degree $1$ operator $D$ on $C(G;E)$, whose action on $0$-forms is given by 
\begin{equation*}
	(D\varepsilon)(g) = \Delta_g \varepsilon_{s(g)} - \varepsilon_{t(g)}
\end{equation*}
for $\varepsilon \in \Gamma(E)$ and $g \in G$, and for $p>0$ given by
\begin{equation}\label{eqn:domega}
\begin{split}
  	 (D\omega) (g_0, \dots, g_p) =& \Delta_{g_0} \omega (g_1, \dots, g_p) + \sum_{i=1}^p (-1)^i \omega(g_0, \dots, g_{i-1} g_i, \dots, g_p) \\
&+ (-1)^{p+1} \omega(g_0, \dots, g_{p-1})   
\end{split}
\end{equation}
for $\omega \in C^p(G;E)$.  
The operator $D$ satisfies the equation $D^2=0$ and the following graded Leibniz identity:
\begin{equation}
D(\omega \star f) = (D \omega) \star f + (-1)^{|\omega|} \omega \star (\delta f). \label{eqn:leibniz}\\
\end{equation}
The cohomology of the complex $(C(G;E), D)$ is known as the smooth groupoid cohomology of $G$ with values in $E$. In the case where $E$ is the trivial real line bundle over $M$ with the trivial representation, then we recover the $\reals$-valued smooth groupoid cohomology.

We define normalized $E$-valued cochains in the same way as with $\reals$-valued cochains. The space of normalized cochains is closed under the action of $D$, and $\omega \star f$ is normalized if $\omega \in C(G;E)$ and $f \in C(G)$ are both normalized.

\subsection{Lie groupoid representations revisited}
In this section, we will proceed in the direction opposite to that of \S\ref{sec:cohomologyrep}; that is, we will begin with an operator $D$ and attempt to construct a representation $\Delta$.

Let $G \arrows M$ be a Lie groupoid, let $E \to M$ be a vector bundle, and let $D$ be a degree $1$ operator on $C(G;E)$ satisfying \eqref{eqn:leibniz}. For any $g \in G$, we may obtain a linear map $\Delta_g : E_{s(g)} \to E_{t(g)}$, given by
\begin{equation}\label{eqn:difftorep}
\Delta_g \varepsilon_{s(g)}= (D \varepsilon) (g) + \varepsilon_{t(g)}	
\end{equation}
for any $\varepsilon \in \Gamma(E)$. Using \eqref{eqn:leibniz}, one can verify that
\[ D(\varepsilon f)(g) + (\varepsilon f)_{t(g)} = (D(\varepsilon )(g) + \varepsilon_{t(g)}) f_{s(g)}\]
for any $f \in C^\infty(M)$. This implies that \eqref{eqn:difftorep} well-defines $\Delta_g$.

We may think of $\Delta: g \mapsto \Delta_g$ as a map from $G$ to the \emph{frame category} $\framecat(E)$ whose set of objects is $M$ and whose morphisms are (not necessarily invertible) linear maps $E_x \to E_y$ for $x,y \in M$. The frame category is a Lie category\footnote{A Lie category is defined in the same way as a Lie groupoid, except without an inverse map.} that contains the frame groupoid as the subcategory consisting of all invertible elements. In fact, the simplest way to prove that the frame groupoid is smooth is to recognize it as an open subset of the frame category. The map $\Delta: G \to \framecat(E)$ is smooth since $D$ is continuous, but in general $\Delta$ will not respect composition. This point motivates the following notion of \emph{quasi-action} \cite{aba-cra:rephomgpd}.

\begin{dfn}
	A \emph{quasi-action} of $G$ on $E$ is a smooth map $\Delta: G \to \framecat(E)$ that respects source and target maps. 
\end{dfn}

\begin{dfn}\label{dfn:unitalflat}
	A quasi-action $\Delta$ is called 
\begin{enumerate}
     \item \emph{unital} if $\Delta_{1_x} = \id$ for all $x \in M$,
     \item \emph{flat} if $\Delta_{g_1} \Delta_{g_2} = \Delta_{g_1 g_2}$ for all $(g_1,g_2) \in G^{(2)}$.
\end{enumerate}
\end{dfn}
Clearly, a flat and unital quasi-action is the same thing as a representation. In particular, if both conditions in Definition \ref{dfn:unitalflat} hold, then the image of $\Delta$ is contained in the frame groupoid of $E$.

\begin{ex}
     To illustrate the notion of quasi-action, we give an example where $G = S^2 \times S^2 \arrows S^2$ is the pair groupoid and $E = TS^2$. Given $(y,x) \in S^2 \times S^2$, we define a map $\Delta_{(y,x)}: T_x S^2 \to T_y S^2$ as follows. Equip $S^2$ with the standard spherical metric, where the distance between two antipodal points is $\pi$. If $x$ and $y$ are antipodal, then $\Delta_{(y,x)}$ is the zero map. Otherwise, $\Delta_{(y,x)}$ is given by parallel transport along the shortest geodesic from $x$ to $y$, together with scalar multiplication by a factor of $(1 + \cos(d(x,y)))/2$, where $d$ is the distance function. The scaling factor ensures that the map $\Delta: (y,x) \mapsto \Delta_{(y,x)}$ is smooth at antipodal pairs. The quasi-action $\Delta$ is unital but not flat. 

     We use this example to emphasize some points about quasi-actions:
\begin{itemize}
     \item The definition of quasi-action allows for the possibility of $\Delta_g$ being degenerate, as is the case for antipodal pairs in the above example. This point is crucial, since if we required every $\Delta_g$ to be nondegenerate (equivalently, if we required the image of $\Delta$ to be in the frame groupoid), then there would be no examples in the case where $G = S^2 \times S^2$ and $E = TS^2$, since such an example would imply the existence of a trivialization of $TS^2$.
     \item This example illustrates the general fact (which we will see in Example \ref{ex:trivial}) that unital quasi-actions always exist for arbitrary $G$ and $E$, although they are not canonical. For example, we can obtain other unital quasi-actions of $S^2 \times S^2$ on $TS^2$ by replacing the above scaling factor by any smooth function $f(y,x)$ on $S^2 \times S^2$ that equals $1$ when $y=x$ and $0$ when $x$ and $y$ are antipodal, such as a bump function supported on a neighborhood of the diagonal submanifold $\{(x,x)\}$. 
\end{itemize}
\end{ex}

So far, we have seen that, given a degree $1$ operator $D$ on $C(G;E)$ satisfying \eqref{eqn:leibniz}, we can obtain a quasi-action $\Delta$ defined by \eqref{eqn:difftorep}. Using \eqref{eqn:leibniz}, one can then show that $D$ must satisfy \eqref{eqn:domega} for $p > 0$. The following lemma, which we leave as an exercise (also see \cite{aba-cra:rephomgpd}), expresses in terms of $D$ the conditions for $\Delta$ to be unital and flat.

\begin{lemma} Let $D$ be a degree $1$ operator on $C(G;E)$ satisfying \eqref{eqn:leibniz}, and let $\Delta$ be the quasi-action given by \eqref{eqn:difftorep}. Then
	\begin{enumerate}
		\item $\Delta$ is flat if and only if $D^2 = 0$.
	       \item $\Delta$ is unital if and only if $D$ preserves the space of normalized cochains.
	\end{enumerate}
\end{lemma}

The following theorem ties together the results from this section and \S\ref{sec:cohomologyrep}.
\begin{thm}\label{thm:repop}
There is a one-to-one correspondence between representations of $G$ on $E$ and degree $1$ operators $D$ on $C(G;E)$ satisfying \eqref{eqn:leibniz}, preserving the space of normalized cochains, and such that $D^2 = 0$.
\end{thm}

\subsection{Representations up to homotopy}

Let $\Etot = \bigoplus E_i$ be a graded vector bundle over $M$.  We consider $C(G; \Etot)$ to be a graded right $C(G)$-module with respect to the total grading:
\[ C(G; \Etot)^p = \bigoplus_{q-r=p} C^q(G; E_r).\]

From the point of view of Theorem \ref{thm:repop}, the following is a natural extension of the notion of representation to the graded setting.

\begin{dfn}[\cite{aba-cra:rephomgpd}]\label{dfn:superrep}
	A \emph{representation up to homotopy} of $G$ on a graded vector bundle $\Etot$ is a continuous degree $1$ operator $\totaldiff$ on $C(G;\Etot)$ satisfying \eqref{eqn:leibniz}, preserving the space of normalized cochains, and such that $\totaldiff^2 = 0$.
\end{dfn}
We stress that this definition of representation up to homotopy agrees with that of \emph{unital} representation up to homotopy in \cite{aba-cra:rephomgpd}.

\subsection{Transformation cochains}

Let $G \arrows M$ be a Lie groupoid, and let $E$ and $C$ be vector bundles over $M$. We define the space of \emph{transformation $p$-cochains} from $E$ to $C$ as $C^p(G;E \to C) \defequal \Gamma\left(\Hom \left((\pi_p^p)^* E, (\pi_0^p)^* C \right)\right)$, where $\pi^p_p: G^{(p)} \to M$ is the identity for $p=0$ and is given by $\pi^p_p(g_1, \dots, g_p) = s(g_p)$ for $p>0$. More concretely, if $\omega \in C^p(G; E \to C)$ and $(g_1, \dots, g_p) \in G^{(p)}$, then $\omega_{(g_1, \dots, g_p)}$ is a linear map from $E_{s(g_p)}$ to $C_{t(g_1)}$.

We note that $C^0(G; E \to C) = \Hom(E,C)$; however, for $p > 0$, $C^p(G; E \to C)$ is different from the space of $p$-cochains with values in $\Hom(E,C)$, except when $M$ is a point.

Let $\omega \in C^p(G; E \to C)$, and let $\varepsilon \in \Gamma(E) = C^0(G;E)$. Define a $p$-cochain $\hat{\omega}(\varepsilon) \in C^p(G;C)$ by
\begin{equation}\label{eq:transhat}
	\hat{\omega}(\varepsilon)(g_1, \dots, g_p) = \omega_{(g_1, \dots, g_p)}(\varepsilon_{s(g_p)})  \in C_{t(g_1)}.
\end{equation}
The map $\varepsilon \mapsto \hat{\omega}(\varepsilon)$ can be extended to a $C(G)$-module morphism $\hat{\omega}: C^\bullet(G;E) \to C^{\bullet + p}(G;C)$. We leave the following proposition as an exercise (see \cite[Lemma 3.10]{aba-cra:rephomgpd}).

\begin{prop}\label{prop:transformation}
	The map $\omega \mapsto \hat{\omega}$ is an isomorphism from $C^p(G; E \to C)$ to the space of $C(G)$-module morphisms from $C^\bullet (G;E)$ to $C^{\bullet+p}(G;C)$.
\end{prop}

Suppose that $G$ is equipped with quasi-actions $\Delta^E$ and $\Delta^C$ on $E$ and $C$, respectively. Then we may define an operator $D$ on $C(G;E \to C)$ by 
\begin{equation}\label{eq:homdiff}
\begin{split}
  	 (D\omega) (g_0, \dots, g_p) =& \Delta^C_{g_0} \circ \omega (g_1, \dots, g_p) + \sum_{i=1}^p (-1)^i \omega(g_0, \dots, g_{i-1} g_i, \dots, g_p) \\
&+ (-1)^{p+1} \omega(g_0, \dots, g_{p-1}) \circ \Delta^E_{g_p}.   
\end{split}
\end{equation}
Via the isomorphism of Proposition \ref{prop:transformation}, this operator is equivalently given by
\begin{equation*}
	\widehat{D \omega} = D^C \circ \hat{\omega} + (-1)^{p+1} \hat{\omega} \circ D^E,
\end{equation*}
where $D^C$ and $D^E$ are the operators on $C(G;C)$ and $C(G;E)$, respectively, corresponding to the two quasi-actions.

In the case where $\Delta^E$ and $\Delta^C$ are both representations, then the operator $D$ in \eqref{eq:homdiff} satisfies $D^2 = 0$, and one can then define the cohomology $H^\bullet(G; E \to C)$.

\subsection{Representations up to homotopy: $2$-term case}\label{sec:2term}

In this paper, we will be concerned primarily with representations up to homotopy on graded vector bundles that are concentrated in degrees $0$ and $1$. In this case, we use the notation $\Etot = E \oplus C[1]$, where $E$ is the degree $0$ part and $C$ is the degree $1$ part. Then
\[ C(G;E \oplus C[1])^p = C^{p}(G;E) \oplus C^{p+1}(G;C).\] 
Any degree $1$ operator $\totaldiff$ on $C(G;E \oplus C[1])$ decomposes as the sum of the following four homogeneous components:
\begin{align*}
	\hat{\boundary} &: C^\bullet(G;C) \to C^\bullet(G;E), \\
	D^C &: C^\bullet(G; C) \to C^{\bullet+1}(G; C), \\
	D^E &: C^\bullet(G; E) \to C^{\bullet+1}(G; E), \\
	\hat{\Omega} &: C^\bullet(G;E) \to C^{\bullet+2}(G;C).
\end{align*}
The Leibniz rule \eqref{eqn:leibniz} for $\totaldiff$ is equivalent to the requirements that
\begin{enumerate}
	\item $D^C$ and $D^E$ satisfy \eqref{eqn:leibniz}, and
\item $\hat{\boundary}$ and $\hat{\Omega}$ are right $C(G)$-module morphisms.
\end{enumerate}
Requirement (1) implies that there are quasi-actions $\Delta^C$ and $\Delta^E$ on $C$ and $E$, respectively, given by the following graded versions of \eqref{eqn:difftorep}:
\begin{align}
 \Delta^C_g \alpha &= -(D^C \alpha) (g) + \alpha_{t(g)}	\label{eqn:difftorepodd} \\
 \Delta^E_g \varepsilon &= (D^E \varepsilon) (g) + \varepsilon_{t(g)} \label{eqn:difftorepeven}
\end{align}
for $\alpha \in \Gamma(C)$ and $\varepsilon \in \Gamma(E)$. The reason for the sign difference between \eqref{eqn:difftorepodd} and \eqref{eqn:difftorepeven} is that the Leibniz rule now incorporates the vector bundle grading. 

Requirement (2) above implies (via Proposition \ref{prop:transformation}) that $\hat{\boundary}$ corresponds to a linear map $\boundary \in \Hom(C,E) = C^0(G;C \to E)$, and that $\hat{\Omega}$ corresponds to a transformation $2$-cochain $\Omega \in C^2(G;E \to C)$.

Next, we will express equation $\totaldiff^2 = 0$ and the property of preserving normalized cochains in terms of $\Delta^C$, $\Delta^E$, $\boundary$, and $\Omega$.

The equation $\totaldiff^2 = 0$ decomposes into the following equations:
\begin{align*}
	D^E \hat{\boundary} + \hat{\boundary} D^C &= 0, \\
(D^C)^2 + \hat{\Omega} \hat{\boundary} &= 0,\\
(D^E)^2 + \hat{\boundary} \hat{\Omega} &= 0,\\
D^C \hat{\Omega} + \hat{\Omega} D^E &= 0.
\end{align*}
These equations respectively translate into the following equations:\begin{align}
\label{eq:4eq1}
	\Delta^E_{g_1} \boundary - \boundary \Delta^C_{g_1} &= 0, \\
\label{eq:4eq2}  \Delta^C_{g_1} \Delta^C_{g_2} - \Delta^C_{g_1 g_2} + \Omega_{g_1,g_2} \boundary &= 0, \\
\label{eq:4eq3}  \Delta^E_{g_1} \Delta^E_{g_2} - \Delta^E_{g_1 g_2} +  \boundary \Omega_{g_1,g_2} &= 0, \\
\label{eq:4eq4}  \Delta^C_{g_1} \Omega_{g_2,g_3} - \Omega_{g_1 g_2, g_3} + \Omega_{g_1, g_2 g_3} - \Omega_{g_1,g_2}\Delta^E_{g_3} &= 0
\end{align}
for $(g_1,g_2,g_3) \in G^{(3)}$. Equation \eqref{eq:4eq1} says that $\Delta_{g_1} = (\Delta^C_{g_1},\Delta^E_{g_1})$ is a chain map on the $2$-term complex $C \tolabel{\boundary} E$. Equations \eqref{eq:4eq2}--\eqref{eq:4eq3} say that $\Omega_{g_1,g_2}$ provides a chain homotopy from $\Delta_{g_1}\Delta_{g_2}$ to $\Delta_{g_1 g_2}$. Equation \eqref{eq:4eq4} is a Bianchi-type identity, saying that $D\Omega = 0$, where $D$ is defined as in \eqref{eq:homdiff}. In particular, in the case where $\Delta^C$ and $\Delta^E$ are genuine representations, then \eqref{eq:4eq4} can be interpreted as a cocycle condition.

The total operator $\totaldiff$ preserves normalized cochains if and only if all of the four components do. For $\boundary$, the property is automatic. For the remaining three components, we obtain the following conditions:
\begin{align}
     &\Delta^C \mbox{ and } \Delta^E \mbox{ are unital,} \label{eq:unit1}\\
&\Omega \mbox{ is normalized.} \label{eq:unit2}
\end{align}

The following theorem summarizes the results from this section:
\begin{thm}\label{thm:2term}
     There is a one-to-one correspondence between representations up to homotopy of a Lie groupoid $G \arrows M$ on a $2$-term graded vector bundle $E \oplus C[1] \to M$ and $4$-tuples $(\boundary, \Delta^C, \Delta^E, \Omega)$, where
\begin{itemize}
     \item $\boundary: C \to E$ is a linear map,
     \item $\Delta^C$ and $\Delta^E$ are unital quasi-actions of $G$ on $C$ and $E$, respectively, and
     \item $\Omega$ is a normalized element of $C^2(G;E \to C)$,
\end{itemize}
satisfying \eqref{eq:4eq1}--\eqref{eq:4eq4}.
\end{thm}

\subsection{Gauge transformations} \label{sec:gauge}

Let $\Etot$ be a graded vector bundle over $M$. There is a natural quotient map $\mu: C(G;\Etot) \to \Gamma(\Etot)$ whose kernel is spanned by all $C^q(G;E_r)$ where $q > 0$.
\begin{dfn}\label{dfn:equivalence}
A \emph{gauge transformation} of $C(G;\Etot)$ is a degree-preserving $C(G)$-module automorphism $T$ of $C(G;\Etot)$, preserving the space of normalized cochains, such that $\mu \circ T = \mu$. Under a gauge transformation, a representation up to homotopy $\totaldiff$ transforms as $\totaldiff' = T \circ \totaldiff \circ T^{-1}$. In this case, we say that $\totaldiff$ and $\totaldiff'$ are \emph{gauge-equivalent}.
\end{dfn}
Gauge-equivalent representations up to homotopy are isomorphic in the sense of \cite{aba-cra:rephomgpd}, but the notion of gauge-equivalence is slightly more refined than that of isomorphism. Specifically, the condition involving $\mu$ in Definition \ref{dfn:equivalence} serves the purpose of restricting attention to isomorphisms that ``cover'' the identity map on $\Gamma(\Etot)$.

In the case where $\Etot = E$ is concentrated in degree $0$, a representation up to homotopy is the same thing as a representation, and there are no nontrivial gauge transformations.

Let $G \arrows M$ be a Lie groupoid, let $E \oplus C[1]$ be a $2$-term graded vector bundle, and consider a normalized transformation $1$-cochain $\sigma \in C^1(G;E \to C)$. The associated operator $\hat{\sigma}: C^\bullet(G;E) \to C^{\bullet+1}(G;C)$ may be viewed as a degree $0$ operator on $C(G;E \oplus C[1])$. Clearly, $\hat{\sigma}^2 = 0$, so the map $1 + \hat{\sigma}$ is invertible with inverse $1 - \hat{\sigma}$. One can easily see that $1 + \hat{\sigma}$ is a gauge transformation. The converse is also true:
\begin{prop}\label{prop:gauge2term}
	Every gauge transformation of $C(G; E \oplus C[1])$ is of the form $1+\hat{\sigma}$ for some normalized $\sigma \in C^1(G;E \to C)$.
\end{prop}
\begin{proof}
Let $T: C(G;E \oplus C[1]) \to C(G;E \oplus C[1])$ be a gauge transformation. Since $T$ is a $C(G)$-module automorphism, it is completely determined by its action on the $0$-cochain spaces $\Gamma(E)$ and $\Gamma(C)$. Since $T$ is degree-preserving, it sends $\Gamma(C)$ to $\Gamma(C)$ and $\Gamma(E)$ to $\Gamma(E) \oplus C^1(G;C)$. The condition $\mu \circ T = \mu$ implies that the action of $T$ on $\Gamma(C)$ is the identity, and that $T(\varepsilon) - \varepsilon$ is in $\ker \mu$ (and therefore must be in $C^1(G;C)$) for $\varepsilon \in \Gamma(E)$. Hence, the map $T-1$ is a $C(G)$-module morphism taking $C^\bullet(G;E)$ to $C^{\bullet+1}(G;C)$ and can be identified via Proposition \ref{prop:transformation} with $\hat{\sigma}$ for some $\sigma \in C^1(G;E \to C)$. Since $T$ preserves the space of normalized cochains, so does $\hat{\sigma}$, implying that $\sigma$ is normalized.
\end{proof}

We will study how $2$-term representations up to homotopy transform under gauge transformations in \S\ref{sec:dependence}; in particular, see equations \eqref{eq:changecomponents}.

%==================================================================================

\section{$\VB$-groupoids}\label{sec:vbg}

In this section, we review the various equivalent definitions of $\VB$-groupoid, the construction of the core of a $\VB$-groupoid, and the notion of horizontal lift. Most of the material in \S\S\ref{sec:def}-\ref{sec:cores} can be found elsewhere (for example, \cite{mac:book}), but we hope that the reader will find our presentation valuable. In \S\ref{sec:formulas}, we present the formulas for the components of the representation up to homotopy arising from the choice of a horizontal lift.

\subsection{The definition of $\VB$-groupoid}\label{sec:def}
Consider a commutative diagram of Lie groupoids and vector bundles as follows:
\begin{equation} \label{diag:vbg}
\vbgs{\Gamma}{E}{G}{M}
\end{equation}
By this we mean that $\Gamma \arrows E$ is a Lie groupoid (with source, target, multiplication, identity, and inverse maps $\tilde s$, $\tilde t$, $\tilde m$, $\tilde 1$, and $\tilde \iota$),  $G \arrows M$ is a Lie groupoid (with source, target, multiplication, identity, and inverse maps $s$, $t$, $m$, $1$, and $\iota$), $\Gamma \to G$ is a vector bundle (with projection map and zero section $\tilde q$ and $\tilde 0$), $E \to M$ is a vector bundle (with projection map and zero section $q$ and $0$) and such that $q \tilde s = s \tilde q$ and $q \tilde t = t \tilde q$.  For the rest of this subsection we will always start with this data.

\begin{dfn} \label{dfn:vbg}
A \emph{$\VB$-groupoid} is a commutative diagram of Lie groupoids and vector bundles like \eqref{diag:vbg} such that the following conditions hold:
\begin{enumerate}
\item  $(\tilde s, s)$ is a morphism of vector bundles.
\item  $(\tilde t, t)$ is a morphism of vector bundles.
\item  $(\tilde q, q)$ is a morphism of Lie groupoids.
\item  \emph{Interchange law:} 
\begin{equation*}
(\gamma_1 + \gamma_3) (\gamma_2 + \gamma_4) = \gamma_1 \gamma_2 + \gamma_3 \gamma_4
\end{equation*}
for any $\gamma_j \in \Gamma$ for which the equation makes sense; specifically,
for any $(\gamma_1, \gamma_2) \in \Gamma^{(2)}$, $(\gamma_3, \gamma_4) \in \Gamma^{(2)}$ such that $\tilde q (\gamma_1) = \tilde q (\gamma_3)$ and $\tilde q (\gamma_2) = \tilde q (\gamma_4)$.
\end{enumerate}
\end{dfn}

\begin{ex}
Let $G \arrows M$ be any Lie groupoid and take $\Gamma = TG$, $E = TM$.  Then \eqref{diag:vbg} is a $\VB$-groupoid where $\tilde s = Ts$, $\tilde t = Tt$, $\tilde m = Tm$, et cetera. This is the tangent prolongation $\VB$-groupoid which, as we will later see, plays the role of the ``adjoint representation'' of $G$.
\end{ex} 

On the surface, our definition of $\VB$-groupoid appears different from the usual ones (e.g.\ \cite{mac:dblie1, mac:book}).  In what follows, we will show that the various definitions are equivalent. More precisely, we will see that the conditions in Definition \ref{dfn:vbg} are equivalent to the requirement that $\Gamma \to G$ be a ``Lie-groupoid object in the category of vector bundles'' or, equivalently, that $\Gamma \arrows E$ be a ``vector-bundle object in the category of Lie groupoids''.

First of all, we remark that the usual definition includes the following technical condition.   Consider the manifold
$E \times_s G := s^* E = \{ (e,g) \in E \times G \; \vert \; q(e) = s(g) \}$
and the map
\begin{equation}  \label{eq:pR} 
\begin{split} 
p^R: \Gamma &\to \;  \; E \times_s G = s^* E\\
\gamma &\mapsto   (\tilde s (\gamma), \tilde q (\gamma))
\end{split}
\end{equation}
The technical condition is that $p^R$ is required to be a surjective submersion.  However, Li-Bland and \v{S}evera showed in \cite[Appendix A]{lib-sev} that this condition is unnecessary.  Specifically, if $\Gamma$ is a commutative diagram like $\eqref{diag:vbg}$ and Condition 1 in Definition \ref{dfn:vbg} is satisfied, then the map $p^R$ is automatically a surjective submersion.

\begin{dfn} \label{dfn:lgocvb}
A \emph{Lie-groupoid object in the category of vector bundles} is a commutative diagram of Lie groupoids and vector bundles like \eqref{diag:vbg} such that
\begin{enumerate}
\item  $(\tilde s, s)$ is a morphism of vector bundles.
\item  $(\tilde t, t)$ is a morphism of vector bundles.
\item $\Gamma^{(2)} \to G^{(2)}$ is a vector bundle with the natural structure.
\item  $(\tilde m, m)$ is a morphism of vector bundles.
\end{enumerate}
\end{dfn}
In Definition \ref{dfn:lgocvb}, Condition 3 is necessary to make sense of Condition 4. Each fiber of $\Gamma^{(2)} \to G^{(2)}$ has a natural vector space structure thanks to Conditions 1 and 2, so only local trivializability needs to be checked for Condition 3.
Note that the conditions of Definition \ref{dfn:lgocvb} imply that the identity $(\tilde 1, 1)$ and inverse $(\tilde \iota, \iota)$ maps are also morphisms of vector bundles.

\begin{dfn}  \label{dfn:vboclg}
A \emph{vector-bundle object in the category of Lie groupoids} is a commutative diagram of Lie groupoids and vector bundles like \eqref{diag:vbg} such that 
\begin{enumerate}
\item  $(\tilde q, q)$ is a morphism of Lie groupoids.
\item $\Gamma \bitimes{\tilde q}{\tilde q} \Gamma \arrows E \bitimes{q}{q} E $ is a Lie groupoid with the natural structure.
\item The addition maps $( +, +)$ are a morphism of Lie groupoids.
\end{enumerate}
\end{dfn}

In Definition \ref{dfn:vboclg}, Condition 2 is necessary to make sense of Condition 3. There is a natural groupoid structure on $\Gamma \bitimes{\tilde q}{\tilde q} \Gamma \arrows E \bitimes{q}{q} E $ thanks to Condition 1, and all the maps involved are smooth, so we only need to check that the source map is a submersion in order to satisfy Condition 2.
Note that the conditions of Definition \ref{dfn:vboclg} imply that the scalar multiplication maps and the zero sections are also morphisms of Lie groupoids.

Definitions \ref{dfn:vbg}, \ref{dfn:lgocvb}, and \ref{dfn:vboclg} are equivalent per the next proposition.

\begin{prop}  \label{prop:eqdefs}
Consider a commutative diagram of Lie groupoids and vector bundles like \eqref{diag:vbg}.  Then the following are equivalent:
\begin{itemize}
\item  $\Gamma$ is a $\VB$-groupoid,
\item $\Gamma$ is a ``Lie-groupoid object in the category of vector bundles'',
\item $\Gamma$ is a ``vector-bundle object in the category of Lie groupoids''. 
\end{itemize}
\end{prop}
\begin{proof}
In Lemma \ref{lem:2TC} below we will show that Condition 3 in Definition \ref{dfn:lgocvb} and Condition 2 in Definition \ref{dfn:vboclg} follow from the property of $p^R$ being a submersion, rendering them completely unnecessary.

Next, we notice that Conditions 3 and 4 in Definition \ref{dfn:vbg} are equivalent to Condition 4 in Definition \ref{dfn:lgocvb}, and that Conditions 1, 2, and 4 in Definition \ref{dfn:vbg} are equivalent to Condition 3 in Definition \ref{dfn:vboclg}.  This is shown by writing down each condition as a commutative diagram.  This completes the proof.
\end{proof}

Finally, we note the following identities which are satisfied by $\VB$-groupoids:
\begin{align*}
\tilde{0}_{g_1} \tilde{0}_{g_2} &= \tilde{0}_{g_1 g_2}, & \tilde{0}_{1_x} &= \tilde{1}_{0_x},
\end{align*}
for all $(g_1,g_2) \in G^{(2)}$ and all $x \in M$. We will use these identities without reference throughout the remainder of the paper.

%=================================================================================
\subsection{Cores and decompositions of a $\VB$-groupoid}  \label{sec:cores}

\subsubsection{The right-core}\label{sec:rightcore}

Let $\Gamma$ be a $\VB$-groupoid as in \eqref{diag:vbg}.   Notice that $s^*E = E \times_s G \to G$ is a vector bundle, and that $p^R$, as defined in \eqref{eq:pR}, is a surjective morphism of vector bundles covering the identity map on $G$.

\begin{dfn}\label{dfn:rightvert}
The kernel of $p^R$, which we will denote $V^R \to G$, is the \emph{right-vertical subbundle} of $\Gamma \to G$.  Equivalently, for every $g \in G$, the fiber $V^R_g$ is the kernel of $\tilde{s}_g : \Gamma_g \to E_{s(g)}$.  Elements in $V^R$ are called \emph{right-vertical} elements of $\Gamma$.
\end{dfn}

\begin{dfn}
The \emph{right-core} of the $\VB$-groupoid \eqref{diag:vbg} is $C^R : = 1^*(V^R)$. In other words, $C^R$ is the restriction of $V^R$ to the units of $G$.
\end{dfn}

For any $(g_1,g_2) \in G^{(2)}$, right-multiplication by $\tilde{0}_{g_2}$ produces a linear isomorphism from $V^R_{g_1}$ to $V^R_{g_1 g_2}$.  In particular, for any $g \in G$, right-multiplication by $\tilde{0}_g$ produces a linear isomorphism from $C^R_{t(g)}$ to $V^R_g$.  Hence, we have a natural isomorphism of vector bundles over $G$ between $V^R$ and $t^* (C^R) = C^R \times_t G $ given by:
\begin{equation}  \label{eq:jR}  
\begin{split} 
j^R: \; t^*( C^R) = C^R \times_t G & \longrightarrow  V^R \subseteq \Gamma \\
(c, g)  \quad  & \mapsto \quad c \cdot \tilde{0}_g
\end{split}
\end{equation}

Now consider the following short exact sequence of vector bundles:
\begin{equation}  \label{eq:ses}
\xymatrix{
t^* (C^R) \ar[r]^{j^R} \ar[d] & \Gamma \ar[r]^{p^R} \ar[d] & s^* E  \ar[d] \\
G \ar[r]^{\id} & G \ar[r]^{\id}  & G
},
\end{equation}
where $p^R$ and $j^R$ are defined by Equations \eqref{eq:pR} and \eqref{eq:jR} respectively.  

We recall that a section of \eqref{eq:ses} is a morphism of vector bundles $h: s^*E \to \Gamma$ such that $p^R \circ h = \id$.  Such a section induces a splitting, that is, an isomorphism of vector bundles $s^* E \oplus t^* (C^R) \iso \Gamma $, given by
\begin{equation}  \label{eq:VH}
\begin{split}
C^R_{t(g)}  \oplus E_{s(g)} & \toiso \Gamma_g,  \\
 (c, e) &\mapsto c \cdot \tilde 0_g + h_g(e),
\end{split}
\end{equation}
where $h_g(e) \defequal h(e,g)$.
Given a choice of a section, the image of $h_g$ is a complement to $V^R_g \subseteq \Gamma_g$.  We will refer to vectors in the image of $h_g$ as \emph{right-horizontal}. 

Of course, the right-horizontal subspaces are noncanonical; however, when $g$ is a unit, there is a natural splitting
\begin{equation}\label{eqn:nattyice}
\Gamma_{1_x}  = C^R_x \oplus \tilde 1 (E_x) \textrm{ for all } x \in M.
\end{equation}
We will restrict our attention to splittings whose restriction to the units coincides with \eqref{eqn:nattyice}.
\begin{dfn}
A \emph{right-horizontal lift} of the $\VB$-groupoid \eqref{diag:vbg} is a section $h: s^*E \to \Gamma$ of \eqref{eq:ses}
such that 
\begin{equation}  \label{eq:unitalhorlift}
h(e, 1_x) = \tilde{1}_e \textrm{ for all } x \in M \textrm{ and } e \in E_x.  
\end{equation}
A \emph{right-decomposition} of the $\VB$-groupoid \eqref{diag:vbg} is a splitting of \eqref{eq:ses}
which coincides with the splitting \eqref{eqn:nattyice} at the units of $G$.
\end{dfn}
Clearly, right-horizontal lifts and right-decompositions are equivalent to each other. Since a right-decomposition is a splitting of a short exact sequence of vector bundles which agrees with a given splitting on an embedded submanifold, the existence of right-decompositions (and hence right-horizontal lifts) follows from a standard partition-of-unity argument.

\begin{ex}
 For the tangent prolongation $\VB$-groupoid $TG$, the right-core consists of vectors at units of $G$ that are tangent to the $s$-fibers. In other words, the right-core is the Lie algebroid $A$ of $G$.  The short exact sequence \eqref{eq:ses} is then
\begin{equation}  \label{eq:sesTG}
 t^*(A) \to TG \to s^*(TM),
\end{equation}
where the first map is given by right-translation and the second map is push-forward by $s$. A right-horizontal lift of $TG$ is the same thing as an Ehresmann connection on $G$, in the sense of \cite{aba-cra:rephomgpd}.
\end{ex}

%--
\subsubsection{The left-core}

By exchanging the roles of source and target, one can similarly define a \emph{left-core}. All the concepts in \S\ref{sec:rightcore} have analogues in this setting. In particular, the analogue of \eqref{eq:ses} is the short exact sequence
\begin{equation} \label{eq:ses2}
\xymatrix{
s^* (C^L) \ar[r]^{j^L} & \Gamma \ar[r]^{p^L}  & t^* E 
},
\end{equation}
\begin{ex}
 For the tangent prolongation $\VB$-groupoid $TG$, the left-core consists of vectors at units of $G$ that are tangent to the $t$-fibers, so the left-core can also be identified with the Lie algebroid of $G$. The left-core and right-core for $TG$ correspond to the two models for the Lie algebroid of $G$, defined via left- and right-invariant vector fields, respectively.
\end{ex}

%--
\subsubsection{Two cores, un c{\oe}ur}

There is a canonical isomorphism between the left- and right-cores of a $\VB$-groupoid. This fact allows us to see $C^L$ and $C^R$ as simply being two different models for a single natural vector bundle $C$. 
\begin{prop}
The involution
\begin{equation*}
F: \; \gamma \in \Gamma \; \mapsto \; -\gamma^{-1} \in \Gamma
\end{equation*}
exchanges the right-core $C^R$ and the left-core $C^L$. 
\end{prop}
We will abuse notation and use $F$ to refer to the restriction of $F$ to either $C^L$ or $C^R$.
Explicitly, the restrictions are given by 
\begin{equation}\label{eq:rightleft}
F(c) =  -c^{-1} = c - \tilde 1_{\tilde t (c)} 
\end{equation}
for $c \in C^R$ and $F(c) = c - \tilde 1_{\tilde s (c)}$ for $c \in C^L$. The map $F$ can be used to transform expressions involving one core into those involving the other.

\begin{ex}
Consider the tangent prolongation groupoid $TG$ in the case where $M$ is a point (so that $G$ is a Lie group). In this case, the left-core and the right-core are both equal to the tangent space at the identity of $G$, and the isomorphism $F$ is the identity map.
\end{ex}

There is a natural correspondence between left- and right-horizontal lifts.  Explicitly, let $h: E \times_t G \to \Gamma$ be a section of \eqref{eq:ses}.  We can associate a section $h' : E \times_s G \to \Gamma$ of \eqref{eq:ses2} by
\[h' (e,g) = \left( h(e, g^{-1}) \right)^{-1}\]

In the remainder of the paper we will, whenever possible, remain model-agnostic and simply refer to the core $C$.  However, when writing specific formulas it will frequently be necessary to choose a model. Unless otherwise specified, we will in these situations take $C$ to mean $C^R$.

%--
\subsubsection{Proof of technical conditions}

We shall now prove the following lemma, which completes the proof of Proposition \ref{prop:eqdefs}. 

\begin{lemma} \label{lem:2TC}
Every $\VB$-groupoid satisfies Condition 3 in Definition \ref{dfn:lgocvb} and Condition 2 in Definition \ref{dfn:vboclg}.
\end{lemma}
\begin{proof}
First, we choose a decomposition, so that we have isomorphisms  $\Gamma \iso s^* (C^L) \oplus t^* E$ and $\Gamma \iso t^* (C^R) \oplus s^* E$.  These isomorphisms allow us to decompose $\Gamma^{(2)}$ as
\[\Gamma^{(2)}  \iso (\pi^2_0)^* (C^R)  \oplus_{G^{(2)}}  (\pi^2_1)^* E  \oplus_{G^{(2)}}  (\pi^2_2)^* (C^L),\]
where $\pi^2_0, \pi^2_1, \pi^2_2 : G^{(2)} \to M$ are the three vertex maps given by $\pi^2_0(g_1, g_2) = t(g_1)$, $\pi^2_1(g_1, g_2) = s(g_1) = t(g_2)$, and $\pi^2_2(g_1,g_2) = s(g_2)$.
This shows that $\Gamma^{(2)} \to G^{(2)}$ is a vector bundle, which is Condition 3 in Definition \ref{dfn:lgocvb}.

Next, we observe that Condition 2 in Definition \ref{dfn:vboclg} reduces to checking that the source map of the groupoid $\Gamma \bitimes{\tilde q}{\tilde q} \Gamma \arrows E \bitimes{q}{q} E $ is a submersion.
 This source map can be written as the following composition:

\begin{tabular}{ccccccc}
$\Gamma \bitimes{\tilde q}{\tilde q} \Gamma$ &
$\longrightarrow$  &
$ E \bitimes{q}{s}  G   \bitimes{\id}{\tilde q}  \Gamma$ &
${\stackrel{\sim}{\longrightarrow}}$ &
$E \bitimes{q}{q \tilde s} \Gamma$ &
$\longrightarrow$  &
$E \bitimes{q}{q}  E$
\\
$(\gamma_1, \gamma_2)$ &
$\mapsto$ &
$(\tilde s (\gamma_1), \tilde q (\gamma_1), \gamma_2)$ &
$\mapsto$ &
$(\tilde s (\gamma_1), \gamma_2)$ &
$\mapsto$ &
$(\tilde s (\gamma_1), \tilde s (\gamma_2) )$
\end{tabular}

The first map in this composition is the fibered product of $p^R$ and the identity, the second map is an isomorphism, the third map is the fibered product of the identity and $\tilde q$.  Hence they are all submersions, and so is their composition.
\end{proof}

%=================================================================================

\subsection{How to obtain a representation up to homotopy from a $\VB$-groupoid}\label{sec:formulas}

Let $\Gamma$ be a $\VB$-groupoid as in \eqref{diag:vbg}.  In this section, we give formulas and some geometric explanation for the components $\boundary, \Delta^C, \Delta^E, \Omega$ that correspond (via Theorem \ref{thm:2term}) to a representation up to homotopy of $G$ on the $2$-term graded vector bundle $E \oplus C[1]$.

Although it is possible to check conditions \eqref{eq:4eq1}--\eqref{eq:unit2} directly (see Example \ref{ex:semi2}), we will not do so. Instead, we will later see in \S\ref{sec:vbsuperrep} that there is a \emph{canonically} defined complex that, given a horizontal lift, can be identified with $C(G;E \oplus C[1])$. In Appendix \ref{appendix:derivation}, we show that the formulas for $\boundary$, $\Delta^C$, $\Delta^E$, and $\Omega$ can be derived by transferring the differential from the canonical complex to $C(G;E \oplus C[1])$ and decomposing into homogeneous components. Conditions \eqref{eq:4eq1}--\eqref{eq:unit2} are then immediate consequences.

\subsubsection{The core-anchor}

The \emph{core-anchor} is a vector bundle morphism $\boundary:C \to E$, given by projection by $\tilde{t}$:
\begin{equation}\label{eq:boundary}
  \boundary c = \tilde{t}(c)
\end{equation}
 for $c \in C^R$. In the case $\Gamma = TG$, the core is the Lie algebroid $A$ of $G$, and $\boundary : A \to TM$ coincides with the anchor map.

\subsubsection{The core quasi-action}
The core quasi-action $\Delta^C$ is given by
\begin{equation}\label{eq:coreconnection}
     \Delta^C_g c = h_g(\tilde{t}(c)) \cdot c \cdot \tilde{0}_{g^{-1}}
\end{equation}
for $c \in C^R_{s(g)}$. This may be interpreted as a conjugation action of $g$ on $c$, in the sense that $h_g(\tilde{t}(c))$ is the unique horizontal element of $\Gamma_g$ by which $c$ can be left-multiplied, and $\tilde{0}_{g^{-1}}$ is the unique horizontal element of $\Gamma_{g^{-1}}$ by which $c$ can be right-multiplied. In particular, if $c$ is in $\ker \boundary$, then $\Delta^C_g c = \tilde{0}_g \cdot c \cdot \tilde{0}_{g^{-1}}$. It is clear from this formula that the induced representation on $\ker \boundary$ is canonical.

In the case $\Gamma = TG$, the right-core $C^R$ corresponds to the ``right-invariant vector field'' model of the Lie algebroid $A$ of $G$. In this model, an element $a \in A_{s(g)}$ is a vector in $T_{1_{s(g)}}G$ that is tangent to the $s$-fiber. Right-translation is well-defined for such vectors, but left-translation is not well-defined unless $a \in \ker \rho$. However, when we have chosen a ``horizontal'' subspace of $T_g G$ that is complementary to the $s$-fiber, then left-translation is possible, giving us a quasi-action of $G$ on $A$.

\subsubsection{The side quasi-action}
The side quasi-action $\Delta^E$ is given by 
\begin{equation}\label{eq:sideconnection}
     \Delta^E_g e = \tilde{t}(h_g(e))
\end{equation}
for $e \in E_{s(g)}$. Geometrically, \eqref{eq:sideconnection} simply says that $e$ is horizontally lifted to $\Gamma_g$ and then projected back to $E$ via $\tilde{t}$. If $h^\prime$ is another horizontal lift, then $h^\prime_g(e) = h_g(e) + c \cdot \tilde{0}_g$ for some $c \in C^R_{t(g)}$. Then $\tilde{t}(h^\prime_g(e)) = \tilde{t}(h_g(e)) + \tilde{t}(c) = \Delta^E_g e + \boundary c$. Thus, although $\Delta^E$ depends on the choice of $h$, the induced representation on $\coker \boundary$ is canonical.

In the case $\Gamma = TG$, we have $E = TM$, where the action of $g \in G$ on a vector $v \in T_{s(g)} M$ is given by horizontally lifting $v$ to a vector in $T_g G$ and then projecting by $Tt$.

\subsubsection{The transformation cochain}

The $2$-cochain $\Omega$ measures the failure of the horizontal lift $h$ to be multiplicative. The precise formula is
\begin{equation}\label{eq:omega}
\begin{split}
         \Omega_{g_1,g_2} e &= \left(h_{g_1g_2}(e) - h_{g_1}(\tilde{t}(h_{g_2}(e))) \cdot h_{g_2}(e) \right) \cdot \tilde{0}_{(g_1 g_2)^{-1}} \\
&= \left(h_{g_1g_2}(e) - h_{g_1}(\Delta^E_{g_2} e) \cdot h_{g_2}(e) \right) \cdot \tilde{0}_{(g_1 g_2)^{-1}}
\end{split}
\end{equation}
for $(g_1,g_2) \in G^{(2)}$ and $e \in E_{s(g_2)}$. We will say more about the geometric meaning of $\Omega$ in the case $\Gamma = TG$ in \S\ref{sec:adjoint}.

\begin{remark}
The reader may verify that, in the case $\Gamma = TG$, the formulas for $\boundary$, $\Delta^C$, $\Delta^E$, and $\Omega$ agree with the components of the adjoint representation up to homotopy as defined in \cite{aba-cra:rephomgpd}. 
\end{remark}

%=================================================================================

\subsection{More examples of $\VB$-groupoids}

An important example of $\VB$-groupoid is the tangent prolongation $TG$, which has already been mentioned many times. In this section, we describe more examples. In particular, we describe the semidirect product of $G$ with a $2$-term representation up to homotopy.

\begin{ex}  \label{ex:trivial}
Let $G \arrows M$ be a Lie groupoid and let $E \to M$ be a vector bundle.  Then, as a vector bundle, let $\Gamma \to G$ be defined as $\Gamma := t^* E \oplus s^* E$; that is
\[\Gamma = \left\{ (e_1,g,e_2) \suchthat e_1 \in E_{t(g)}, \, e_2 \in E_{s(g)}\right\}.\]
Then $\Gamma$ is the total space for a $\VB$-groupoid, with source, target, and multiplication maps given by
\begin{align*}
\tilde s(e_1, g, e_2) &= e_2,\\
\tilde t(e_1, g, e_2) &= e_1,\\
(e_1, g_1, e_2) \cdot (e_2, g_2, e_3) &= (e_1, g_1g_2, e_3).
\end{align*}

In this case, $C = E$ and $\boundary$ is the identity.  There is a one-to-one correspondence between horizontal lifts $h$ and unital quasi-actions $\Delta$ on $E$ given by
\[h_g(e) = (\Delta_g(e), g, e).\]
Given a horizontal lift, the resulting representation up to homotopy has side and core quasi-actions both equal to $\Delta$, with $\Omega$ being the ``curvature'' of $\Delta$, given by
\begin{equation*}
	\Omega_{g_1,g_2} e = \Delta_{g_1 g_2} e - \Delta_{g_1} \Delta_{g_2} e.
\end{equation*}

For the purposes of representation theory, we argue that the $\VB$-groupoid $t^*E \oplus s^*E$ plays the role of the ``trivial'' representation of $G$ on $E$, since it contains no additional information beyond the Lie groupoid $G \arrows M$ and the vector bundle $E \to M$.
\end{ex}

\begin{ex}[Semidirect product]\label{ex:semi2}
Let $G \arrows M$ be a Lie groupoid, and let $\totaldiff$ be a representation up to homotopy of $G$ on $E \oplus C[1]$. Let $(\boundary, \Delta^C, \Delta^E, \Omega)$ be the $4$-tuple corresponding to $\totaldiff$ via Theorem \ref{thm:2term}.

Let
\[\Gamma = t^* C \oplus s^*E = \{(c,g,e) \suchthat c \in C_{t(g)}, \, e \in E_{s(g)}\}. \]
We endow $\Gamma$ with a Lie groupoid structure over $E$, defined as follows. The source and target maps $\tilde{s}, \tilde{t}: \Gamma \to E$ are given by
\begin{align}
	\tilde{s}(c,g,e) &= e, \label{eq:tildesdec} \\
	\tilde{t}(c,g,e) &= \boundary c + \Delta^E_g e. \label{eq:tildetdec}
\end{align}
It is clear that $\tilde{s}$ is a submersion.  Multiplication for compatible pairs is given by
\begin{equation}  \label{eq:productdec}
(c_1, g_1, e_1) \cdot (c_2, g_2, e_2)  = \left(c_1 + \Delta^C_{g_1} c_2 - \Omega_{g_1, g_2} e_2,  g_1g_2,  e_2 \right).
\end{equation}
For $e \in E_x$, the identity over $e$ is $\tilde 1_e = (0,1_x, e)$.  Inverses are given by
\begin{equation*}
(c,g,e)^{-1} =  \left(  -\Delta^C_{g^{-1}} c  +  \Omega_{g^{-1}, g} e   ,  g^{-1} ,  \boundary c + \Delta^E_g e \right).
\end{equation*}

The maps $\tilde{s}$, $\tilde{t}$, and $\tilde{m}$ are clearly linear, and the groupoid axioms can be verified by direct computation. We point out the following:
\begin{itemize}
\item  The condition $\tilde t (\gamma_1 \cdot \gamma_2) = \tilde t (\gamma_1)$ is equivalent to  \eqref{eq:4eq1} and \eqref{eq:4eq3}.
\item  Associativity of the product is equivalent to \eqref{eq:4eq2} and \eqref{eq:4eq4}.
\item  Equations \eqref{eq:unit1} and \eqref{eq:unit2} are equivalent to the fact that the identity map $\tilde 1$ we defined is indeed an identity.
\end{itemize}

The existence of horizontal lifts means that any $\VB$-groupoid can be identified with a semidirect product, albeit noncanonically. Using this, one could give a direct proof of the fact that the structures defined in \ref{sec:formulas} satisfy the axioms of a representation up to homotopy.
\end{ex}

\section{Dual of a $\VB$-groupoid}\label{sec:dual}

Consider a $\VB$-groupoid $\Gamma$ like in \eqref{diag:vbg}, with core $C$, and let $\Gamma^* \to G$ be the dual vector bundle to $\Gamma \to G$.  Then \cite{mac:book, pradines2} there is an associated \emph{dual $\VB$-groupoid}
\begin{equation}\label{eq:dual}
 \vbgs{\Gamma^*}{C^*}{G}{M}
\end{equation}
In the case $\Gamma = TG$, the dual $\VB$-groupoid is the cotangent prolongation $T^*G \arrows A^*$.

The dual construction plays a significant role in the definition of the canonical $\VB$-groupoid complex in \S\ref{sec:vbsuperrep}, and we will briefly recall the formulas for the structure maps for later use. We refer the reader to \cite{mac:book} for details and proofs.

The source and target $\du{s}, \du{t} : \Gamma^* \arrows C^*$ are defined as follows. Let $\xi \in \Gamma_g^*$.  Then $\du s (\xi) \in C^*_{s(g)}$  and $\du t (\xi) \in C^*_{t(g)}$ are given by
\begin{align}  \label{eq:sdu}
 \pair{\du s (\xi)}{ c_1 } &= - \pair{ \xi} {\tilde 0_g \cdot c_1^{-1} },\\ \label{eq:tdu}
 \pair{\du t (\xi)}{ c_2 } &=  \pair{ \xi} {c_2 \cdot \tilde 0_g }
\end{align}
for all $c_1 \in C_{s(g)}$ and $c_2 \in C_{t(g)}$. Here and in the following, $\pair{}{}$ denotes the pairing of a vector space and its dual.

The formulas \eqref{eq:tdu}--\eqref{eq:sdu} can be derived by requiring that the short exact sequences
\begin{equation*}
\xymatrix{
s^* (E^*) \ar[r] & \Gamma^* \ar[r] & t^* (C^*)  
},
\end{equation*}
\begin{equation*}
\xymatrix{
t^* (E^*) \ar[r] & \Gamma^* \ar[r] & s^* (C^*)  
},
\end{equation*}
are dual to \eqref{eq:ses} and \eqref{eq:ses2}, respectively. In particular, we note that, for any $g \in G$, $\ker \du{t} \subseteq \Gamma^*_g$ is the annihilator of $\ker \tilde{s} \subseteq \Gamma_g$, and similarly with $s$ and $t$ reversed.

Let $(\xi_1, \xi_2) \in (\Gamma^*)^{(2)}$, where $\xi_i \in \Gamma^*_{g_i}$.  Under the assumption that $\du s(\xi_1) = \du t (\xi_2)$, the product $\xi_1 \cdot \xi_2 \in \Gamma^*_{g_1 g_2}$ is given by the formula
\begin{equation} \label{eq:defmult}
\pair{\xi_1 \cdot \xi_2}{\gamma_1 \cdot \gamma_2}  =  \pair{\xi_1}{\gamma_1} + \pair{\xi_2}{\gamma_2}
\end{equation}
for $\gamma_i \in \Gamma_{g_i}$. The formula \eqref{eq:defmult} can be interpreted as saying that the graph of multiplication in $\Gamma^*$ is, up to sign, the annihilator of the graph of multiplication in $\Gamma$.

%====
\section{$\VB$-groupoids and representations up to homotopy}\label{sec:vbsuperrep}

In this section, we introduce the canonical $\VB$-groupoid complex, and we show that a choice of horizontal lift induces a decomposition of the complex into a $2$-term representation up to homotopy. Different choices of horizontal lift lead to gauge-equivalent representations up to homotopy. In this sense, we can think of $2$-term representations up to homotopy as simply being manifestations of the $\VB$-groupoid complex.

\subsection{$\VB$-groupoid cohomology}

Consider a $\VB$-groupoid $\Gamma$ as in \eqref{diag:vbg} with dual $\VB$-groupoid $\Gamma^*$. Let $\left( C^\bullet (\Gamma^*), \du{\delta} \right)$ be the complex of smooth groupoid cochains associated to the Lie groupoid $\Gamma^* \arrows C^*$. There is a natural subcomplex $C_\lin(\Gamma^*)$, whose $p$-cochains are functions of $(\Gamma^*)^{(p)}$ that are linear over $G^{(p)}$. We call $\left(C_\lin^\bullet(\Gamma^*), \du{\delta}\right)$ the complex of \emph{linear cochains} for the dual $\VB$-groupoid $\Gamma^*$.

\begin{dfn}\label{dfn:strongleft}
 A linear $p$-cochain $\varphi \in C^p_\lin(\Gamma^*)$ is called \emph{left-projectable} if
\begin{enumerate}
\item  
  $\varphi(\du{0}_g, \xi_1, \dots, \xi_{p-1}) = 0$, and
\item 
 $\varphi(\du{0}_g \cdot \xi_1, \dots, \xi_p) = \varphi(\xi_1, \dots, \xi_p)$,
\end{enumerate}
for all $(\xi_1, \dots, \xi_p) \in (\Gamma^*)^{(p)}$ and $g \in G$ such that $\du{t}(\xi_1) = 0_{s(g)}$.
\end{dfn}
The first condition in Definition \ref{dfn:strongleft} implies that $\varphi(\xi_1, \dots, \xi_p)$ only depends on $\xi_1$ and $\du{q}(\xi_i)$ for $2\leq i \leq p$. The second condition is a left-invariance condition for the dependence on $\xi_1$. When $p=0$, both conditions are vacuous, so the space of left-projectable $0$-cochains is $C^0_\lin(\Gamma^*) = \Gamma(C)$. The space of left-projectable $1$-cochains consists of sections $X$ of $\Gamma$ that project via $\tilde{s}$ to a section of $E$ (see Proposition \ref{prop:leftproj} below).

It follows directly from the definition of the coboundary operator $\du{\delta}$ that the left-projectable cochains form a subcomplex of $C_\lin(\Gamma^*)$. This fact allows us to make the following definition.

\begin{dfn}\label{dfn:vbcohomology}
 The \emph{$\VB$-groupoid complex} $\left( C^\bullet_\VB(\Gamma), \du{\delta} \right)$ of $\Gamma$ is the subcomplex of left-projectable cochains in $C\lin(\Gamma^*)$. The \emph{$\VB$-groupoid cohomology} of $\Gamma$ is the cohomology of the $\VB$-groupoid complex.
\end{dfn}

In \S\ref{sec:vbdecomp}, we will see that a choice of decomposition allows us to identify $C_\VB(\Gamma)$ with $C(G;E\oplus C[1])$, and that under this identification the differential $\du{\delta}$ defines a representation up to homotopy of $G$ on $E\oplus C[1]$. For now, we make the following observation, which says that $C_\VB(\Gamma)$, like $C(G;E\oplus C[1])$, is a right $C(G)$-module for which the coboundary operator satisfies the graded Leibiz rule. 

\begin{lemma} \label{lemma:strongleft}
Let $\varphi \in C^p_\lin(\Gamma^*)$ be left-projectable, and let $f \in C^q(G)$ be viewed as a fiberwise-constant element of $C^q(\Gamma^*)$. Then
\begin{enumerate}
\item $\varphi \star f$ is linear and left-projectable, and 
\item $\du{\delta}(\varphi \star f) = (\du{\delta}\varphi) \star f + (-1)^p \varphi \star (\delta f)$.
\end{enumerate}
\end{lemma}

\subsection{$\VB$-groupoid cohomology and decompositions}\label{sec:vbdecomp}

Let $\Gamma$ be a $\VB$-groupoid as in \eqref{diag:vbg}, and let $\varphi \in C_\VB^p(\Gamma)$ be a $\VB$-groupoid cochain. We define an associated map $\hat{\varphi}: G^{(p)} \to \Gamma$, where $\hat{\varphi}_{(g_1,\dots,g_p)} \in \Gamma_{g_1}$, by the equation
\begin{equation}\label{eq:vbhat}
 \pair{\xi_1}{\hat{\varphi}_{(g_1,\dots,g_p)}} = \varphi(\xi_1, \xi_2, \dots, \xi_p)
\end{equation}
for any $(\xi_2, \dots, \xi_p) \in (\Gamma^*)^{(p-1)}$ where $\du{q}(\xi_i) = g_i$. Condition (1) of Definition \ref{dfn:strongleft} implies that $\hat{\varphi}$ is well-defined and that $\varphi$ is completely determined by $\hat{\varphi}$. The following lemma examines the implications of Condition (2).

\begin{lemma}\label{lemma:hatleft}
For $\varphi \in C_\VB^p(\Gamma)$, let $\hat{\varphi}$ be defined as above. Then $\tilde{s}(\hat{\varphi}_{(g_1, g_2, \dots, g_p)}) = \tilde{s}(\hat{\varphi}_{(1_{t(g_2)}, g_2, \dots, g_p)})$ for all $(g_1, \dots, g_p) \in G^{(p)}$. 
\end{lemma}
\begin{proof}
 Expressed in terms of $\hat{\varphi}$, Condition (2) in Definition \ref{dfn:strongleft} says
\begin{equation}\label{eq:hatleft}
 \pair{\du{0}_{g_0} \cdot \xi}{\hat{\varphi}_{(g_0 g_1, g_2, \dots, g_p)}} = \pair{\xi}{\hat{\varphi}_{(g_1, g_2, \dots, g_p)}}
\end{equation}
for all $(g_0, \dots, g_p) \in G^{(p+1)}$ and $\xi \in \Gamma^*_{g_1}$ such that $\du{t}(\xi) = 0_{t(g_1)}$. Pick any $\gamma \in \Gamma_{g_0}$ such that $\tilde{t}(\gamma) = \tilde{t}(\hat{\varphi}_{(g_0 g_1, g_2, \dots, g_p)})$. Then, using \eqref{eq:defmult} and the fact that $\pair{\du{0}_{g_0}}{\gamma} = 0$, we can rewrite the left side of \eqref{eq:hatleft} as $\pair{\du{0}_{g_0} \cdot \xi}{\gamma \cdot \gamma^{-1} \cdot \hat{\varphi}_{(g_0 g_1, g_2, \dots, g_p)}}=\pair{\xi}{\gamma^{-1} \cdot \hat{\varphi}_{(g_0 g_1, g_2, \dots, g_p)}}$. Thus, we have that
\begin{equation}\label{eq:hatleft2}
  \pair{\xi}{\gamma^{-1} \cdot \hat{\varphi}_{(g_0 g_1, g_2, \dots, g_p)} - \hat{\varphi}_{(g_1, g_2, \dots, g_p)}}=0.
\end{equation}
In other words, $\gamma^{-1} \cdot \hat{\varphi}_{(g_0 g_1, g_2, \dots, g_p)} - \hat{\varphi}_{(g_1, g_2, \dots, g_p)}$ is in the annihilator of $\ker \du{t} \subseteq \Gamma^*_{g_1}$, and therefore satisfies $\tilde{s}(\gamma^{-1} \hat{\varphi}_{(g_0 g_1, g_2, \dots, g_p)} - \hat{\varphi}_{(g_1, g_2, \dots, g_p)}) = 0$. We obtain the desired result by setting $g_0 = g_1^{-1}$.
\end{proof}

Conversely, given a map $\hat{\varphi}: G^{(p)} \to \Gamma$, where $\hat{\varphi}_{(g_1,\dots,g_p)} \in \Gamma_{g_1}$, satisfying the equation in Lemma \ref{lemma:hatleft}, we may define a linear cochain $\varphi \in C^p_\lin(\Gamma^*)$ by \eqref{eq:vbhat}, which will satisfy Definition \ref{dfn:strongleft}. In other words, we have the following result:

\begin{prop}\label{prop:leftproj}
The map $\varphi \mapsto \hat{\varphi}$ given by equation \eqref{eq:vbhat} is a bijection from $C^p_\VB(\Gamma)$ to the space of maps $\hat{\varphi}: G^{(p)} \to \Gamma$, where $\hat{\varphi}_{(g_1,\dots,g_p)} \in \Gamma_{g_1}$, such that $\tilde{s}(\hat{\varphi}_{(g_1, g_2, \dots, g_p)}) = \tilde{s}(\hat{\varphi}_{(1_{t(g_2)}, g_2, \dots, g_p)})$ for all $(g_1, \dots, g_p) \in G^{(p)}$. 
\end{prop}

Now suppose that $\Gamma$ is equipped with a horizontal lift $h: s^*E \to \Gamma$. Given $\varphi \in C_\VB^p(\Gamma)$, we may then decompose $\hat{\varphi}$ as in \eqref{eq:VH} to obtain $\varphi^E \in C^{p-1}(G;E)$ and $\varphi^C \in C^p(G;C)$, given by the equation
\begin{equation}\label{eq:hatce}
 \hat{\varphi}_{(g_1, \dots, g_p)} = h_{g_1}(\varphi^E_{(g_2, \dots, g_p)}) + \varphi^C_{(g_1, \dots, g_p)} \cdot \tilde{0}_{g_1}.
\end{equation}
Note that $\varphi^E$ does not depend on $g_1$, as a consequence of Lemma \ref{lemma:hatleft}.

We view the pair $(\varphi^E, \varphi^C)$ as an element of $C^{p-1}(G;E) \oplus C^p(G;C) = C(G;E \oplus C[1])^{p-1}$. The following result is immediate.
\begin{thm}\label{thm:isomorphism}
 The map $\vbiso_h : \varphi \mapsto (\varphi^E,\varphi^C)$ is an isomorphism of graded right $C(G)$-modules from $C_\VB(\Gamma)[1]$ to $C(G;E \oplus C[1])$.
\end{thm}
The isomorphism $\vbiso_h$ in Theorem \ref{thm:isomorphism} depends on the choice of horizontal lift and is therefore noncanonical. However, given such a choice, we may use $\vbiso_h$ to transfer the coboundary operator $\du{\delta}$ on $C_\VB(\Gamma)$ to an operator $\totaldiff_h$ on $C(G;E \oplus C[1])$. The operator $\totaldiff_h$ satisfies the Leibniz rule as a result of Lemma \ref{lemma:strongleft}, and it squares to zero and preserves normalized cochains since $\du{\delta}$ does. Thus we have the following:

\begin{cor}\label{cor:superreph}
 The operator $\totaldiff_h \defequal \vbiso_h \circ (-\du{\delta}) \circ \vbiso_h^{-1}$ on $C(G;E \oplus C[1])$ is a representation up to homotopy of $G$ on the $2$-term graded vector bundle $E \oplus C[1]$.
\end{cor}
The minus sign in the definition of $\totaldiff_h$ arises from the fact that the isomorphism $\vbiso_h$ involves a shift in grading.

In Appendix \ref{appendix:derivation}, we show that the components $\boundary$, $\Delta^C$, $\Delta^E$, and $\Omega$ of the representation up to homotopy $\totaldiff_h$ agree with the formulas given in \S\ref{sec:formulas}, giving us the following:

\begin{cor}\label{cor:whatitis}
The $4$-tuple $(\Delta^C, \Delta^E, \boundary, \Omega)$, as given by formulas \eqref{eq:boundary}--\eqref{eq:omega}, satisfies \eqref{eq:4eq1}--\eqref{eq:unit2} and therefore defines a representation up to homotopy.
\end{cor}

%==================================================================================

\subsection{Dependence of the representation up to homotopy on the decomposition}  \label{sec:dependence}

Let $\Gamma$ be a $\VB$-groupoid as in \eqref{diag:vbg}. We would like to determine how the representation up to homotopy $\totaldiff_h$ changes under a change of horizontal lift.

Let $h$ and $\nue{h}$ be two horizontal lifts.   Let $g \in G$ and let $e \in E_{s(g)}$.  Then $\tilde s(h_g(e)) = \tilde s(\nue{h}_g(e))$, so $h_g(e) - \nue{h}_g(e)$ is vertical. We write
\begin{equation}  \label{eq:newh}
\nue{h}_g(e) = h_g(e) + \sigma_g(e) \cdot \tilde 0_g
\end{equation}
for a unique element $\sigma_g(e) \in C_{t(g)}$.  We may view $\sigma$ as a normalized element of $C^1 (G; E \to C)$.  Conversely, given a horizontal lift $h$ and a normalized transformation $1$-cochain $\sigma \in C^1 (G; E \to C)$, we can define a new horizontal lift $\nue{h}$ by \eqref{eq:newh}.  Thus, we have proven the following.
\begin{lemma}
The space of horizontal lifts is an affine space modeled on the normalized subspace of $C^1 (G; E \to C)$.
\end{lemma}

Suppose that $h$ and $\nue h$ are two horizontal lifts related by $\sigma$ via \eqref{eq:newh}.  From \eqref{eq:hatce} and \eqref{eq:newh}, we can see that the associated ``chart transformation'' on $C(G;E \oplus C[1])$ is given by
\begin{equation}\label{eq:transformation}
	\vbiso_{\nue{h}} \circ \vbiso_{h}^{-1} = 1 - \hat{\sigma},
\end{equation}
where $\hat{\sigma}$ is the operator associated to $\sigma$ (see Proposition \ref{prop:transformation}). Therefore,
 \begin{equation}  \label{eq:changetotaldiff}
\totaldiff_{\nue{h}} = (1 - \hat{\sigma}) \circ \totaldiff_{h} \circ (1 + \hat{\sigma}).
\end{equation}
In light of Proposition \ref{prop:gauge2term}, we see that $\totaldiff_{\nue{h}}$ and $\totaldiff_h$ are gauge-equivalent, and that every element of the gauge-equivalence class of $\totaldiff_h$ appears as $\totaldiff_{\nue{h}}$ for some choice of $\nue{h}$. We summarize the result as follows:

\begin{thm} \label{thm:equivalence}
Let $\Gamma$ be a $\VB$-groupoid as in \eqref{diag:vbg}. The set of all representations up to homotopy $\totaldiff_h$ arising from $\Gamma$ for different choices of horizontal lift $h$ is equal to exactly one gauge-equivalence class of representations up to homotopy of $G$ on $C \oplus E[1]$.
\end{thm}

We can expand \eqref{eq:changetotaldiff} to obtain the following gauge transformation formulas:
\begin{equation}  \label{eq:changecomponents}
\begin{split}
\nue{\boundary} & = \boundary, \\
\nue{\Delta}^C_g & = \Delta^C_g + \sigma_g \boundary, \\
\nue{\Delta}^E_g & = \Delta^E_g + \boundary \sigma_g,  \\
\nue{\Omega}_{g_1, g_2} & = \Omega_{g_1, g_2} - \sigma_{g_1} \Delta^E_{g_2} - \Delta^C_{g_1} \sigma_{g_2} + \sigma_{g_1 g_2} - \sigma_{g_1} \boundary \sigma_{g_2}
\end{split}
\end{equation}

%==================================================================================
\section{The moduli space of $\VB$-groupoids}\label{sec:moduli}

\subsection{The relationship between $\VB$-groupoids and representations up to homotopy}  \label{sec:vbgfromrep}

Corollary \ref{cor:whatitis} tells us that \eqref{eq:boundary}--\eqref{eq:omega} give a well-defined map taking $\VB$-groupoids equipped with horizontal lifts to $2$-term representations up to homotopy. This map is inverted by the semidirect product construction in Example \ref{ex:semi2}. Thus we have the following:

\begin{thm}  \label{thm:big}
There is a one-to-one correspondence between $2$-term representations up to homotopy and $\VB$-groupoids equipped with horizontal lifts.
\end{thm}

Together, Theorems \ref{thm:equivalence} and \ref{thm:big} imply the following:
\begin{cor} \label{cor:big}
	There is a one-to-one correspondence between isomorphism classes of $2$-term representations up to homotopy and isomorphism classes of $\VB$-groupoids.
\end{cor}

%====
\subsection{Classification of regular $2$-term representations up to homotopy}   \label{sec:classification}

Recall from \S\ref{sec:formulas} that the core-anchor $\boundary: C \to E$ associated to a $\VB$-groupoid \eqref{diag:vbg} is independent of the choice of a horizontal lift.

\begin{dfn}  A $\VB$-groupoid (or a $2$-term representation up to homotopy) is called \emph{regular} if its core-anchor $\boundary$ has constant rank.
\end{dfn}
Clearly, a $2$-term representation up to homotopy is regular if and only if it corresponds, via Theorem \ref{thm:big}, to a regular $\VB$-groupoid.

As we have seen in \S\ref{sec:formulas}, one can recover from a $\VB$-groupoid canonical representations of $G$ on $K \defequal \ker \boundary$ and $\nu \defequal \coker \boundary$, but in general these bundles are singular. If the $\VB$-groupoid is regular, then $K$ and $\nu$ are vector bundles. 

In this section, we classify regular $\VB$-groupoids up to isomorphism. Per Corollary \ref{cor:big}, such a classification is equivalent to a classification of the moduli space of regular $2$-term representations up to homotopy. We begin by considering two special cases of regular $\VB$-groupoids.

%====
\subsubsection{$\VB$-groupoids of type $1$}  \label{sec:type1}

\begin{dfn} A $\VB$-groupoid (or a $2$-term representation up to homotopy) is \emph{of type $1$} if its core-anchor $\boundary$ is an isomorphism.
\end{dfn}
We consider how conditions \eqref{eq:4eq1}--\eqref{eq:unit2} for the $4$-tuple $(\boundary, \Delta^C, \Delta^E, \Omega)$ specialize in the type $1$ case. We may assume that $C=E$ and $\boundary = 1$. Then, from \eqref{eq:4eq1}--\eqref{eq:4eq3} and \eqref{eq:unit1}, we have that $\Delta^C = \Delta^E$, where $\Delta^E$ is a unital quasi-action, and that $\Omega_{g_1,g_2}$ is equal to the ``curvature'' $\Delta^E_{g_1 g_2} - \Delta^E_{g_1} \Delta^E_{g_2}$ of $\Delta^E$. Then \eqref{eq:4eq4} and \eqref{eq:unit2} are automatically satisfied.

The representations up to homotopy of the type that we have just described are exactly those that arise from $\VB$-groupoids of the form in Example \ref{ex:trivial}, which do not include any more information than the Lie groupoid $G$ and the vector bundle $E$. Using Corollary \ref{cor:big}, we conclude the following.
\begin{prop}\label{prop:type1}
	Every $\VB$-groupoid of type $1$ is isomorphic to a $\VB$-groupoid of the form in Example \ref{ex:trivial}.
\end{prop}

%====
\subsubsection{$\VB$-groupoids of type 0}  \label{sec:type0} 

\begin{dfn} A $\VB$-groupoid (or a $2$-term representation up to homotopy) is \emph{of type $0$} if its core-anchor $\boundary$ is the zero map.
\end{dfn}
In the type $0$ case, $\boundary=0$ and \eqref{eq:4eq1} holds automatically. Equations \eqref{eq:4eq2}, \eqref{eq:4eq3}, and \eqref{eq:unit1} say that $\Delta^C$ and $\Delta^E$ are representations of $G$ on $C$ and $E$, respectively. Then we may interpret \eqref{eq:4eq4} and \eqref{eq:unit2} as saying that $\Omega \in C^2(G; E \to C)$ is normalized and closed with respect to the differential $D$ in \eqref{eq:homdiff}. 

Next, we consider how a type $0$ representation up to homotopy transforms under a gauge transformation. From \eqref{eq:changecomponents}, we see that the representations $\Delta^C$ and $\Delta^E$ are invariant, while $\Omega$ changes by an exact term:
\begin{equation} \label{eqn:dsigma}
\nue{\Omega} = \Omega - D\sigma
\end{equation}
\begin{prop}\label{prop:type0}
\begin{enumerate}
	\item A type $0$ representation up to homotopy of $G$ on $E \oplus C[1]$ is given by a triple $(\Delta^C, \Delta^E, \Omega)$, where $\Delta^C$ and $\Delta^E$ are representations of $G$ on $C$ and $E$, respectively, and $\Omega \in C^2(G;E \to C)$ is a normalized cocycle.
	\item Two such triples $(\Delta^C, \Delta^E, \Omega)$ and $(\nue{\Delta}^C, \nue{\Delta}^E, \nue{\Omega})$ are gauge-equivalent if and only if $\Delta^C = \nue{\Delta}^C$, $\Delta^E = \nue{\Delta^E}$, and $\Omega$ is cohomologous to $\nue{\Omega}$.
\end{enumerate}
\end{prop}
\begin{cor}  \label{cor:classtype0}
Type $0$ $\VB$-groupoids like \eqref{diag:vbg} are classified up to isomorphism by triples $(\Delta^C, \Delta^E, [\Omega])$, where $\Delta^C$ and $\Delta^E$ are representations of $G$ on $C$ and $E$, respectively, and $[\Omega]$ is a cohomology class in $H^2(G; E \to C)$.
\end{cor}

\begin{remark}
The cochains $\Omega$ appearing in the above analysis are always normalized. The conclusion in Corollary \ref{cor:classtype0} uses the fact that the complex $C(G; E \to C)$ retracts to the subcomplex of normalized cochains. The proof given by Eilenberg and Maclane \cite{eil-mac:1} in the case of group cohomology can be easily extended to the present context.
\end{remark}

%====
\subsubsection{The general case}

Given two $\VB$-groupoids
\begin{equation*}
\vbgs{\Gamma_0}{E_0}{G}{M} 
\quad \quad
\vbgs{\Gamma_1}{E_1}{G}{M} 
\end{equation*}
over the same Lie groupoid $G$, we can form the \emph{direct sum} $\VB$-groupoid
\begin{equation*}
\vbgs{\Gamma_0 \oplus \Gamma_1}{E_0 \oplus E_1}{G}{M} 
\end{equation*}
Note that the core of $\Gamma_0 \oplus \Gamma_1$ is the direct sum of the cores of $\Gamma_0$ and $\Gamma_1$.

\begin{lemma} \label{lemma:0and1}
Given a regular $\VB$-groupoid $\Gamma$, there exist unique (up to isomorphism) $\VB$-groupoids $\Gamma_0$ of type $0$, and $\Gamma_1$ of type $1$, such that $\Gamma$ is isomorphic to $\Gamma_0 \oplus \Gamma_1$.
\end{lemma}
\begin{proof}
Let $\Gamma$ be a regular $\VB$-groupoid as in \eqref{diag:vbg}.  Let $K \defequal \ker \boundary$, $F \defequal \im \boundary$, and $\nu \defequal \coker \boundary$. As a result of the regularity condition, we have that $K$, $F$, and $\nu$ are all vector bundles that fit into the short exact sequences
\begin{equation}\label{eqn:knushort}
\begin{split}
    &\xymatrix{K \ar[r] & C \ar[r] & F}, \\
&\xymatrix{F \ar[r] & E \ar[r] & \nu}.
\end{split}
\end{equation}
We choose splittings of the sequences \eqref{eqn:knushort}, giving isomorphisms
\begin{equation} \label{eqn:1TI}
\begin{split} 
C & \approx K \oplus F \\
E & \approx \nu \oplus F 
\end{split}
\end{equation}
Next, we make a choice of horizontal lift for $\Gamma$.  Such a choice gives us a representation up to homotopy $(\boundary, \Delta^C, \Delta^E, \Omega)$, which completely describes the $\VB$-groupoid structure of $\Gamma$.  Each of the components may be written in ``block-matrix'' form with respect to the direct sums in \eqref{eqn:1TI}.  In particular, the block form of $\boundary$ is $\left(\begin{smallmatrix}0&0\\0&1\end{smallmatrix}\right)$. By \eqref{eq:4eq1}, the block forms of $\Delta^C$  and $\Delta^E$ are of the form
\begin{align*}
     \Delta^C &= \begin{pmatrix} \Delta^K & \Lambda^C \\ 0 & \Delta^F \end{pmatrix} &
     \Delta^E &= \begin{pmatrix} \Delta^\nu & 0 \\ \Lambda^E & \Delta^F \end{pmatrix}.
\end{align*}
Here, $\Delta^K$ and $\Delta^\nu$ are the canonical representations of $G$ on $K$ and $\nu$, respectively, and $\Delta^F$ is a quasi-action on $F$ that depends on the choice of horizontal lift. 

We can block-diagonalize $\Delta^C$ and $\Delta^E$ by setting
\begin{equation*}
\sigma = \begin{pmatrix} 0 & -\Lambda^C \\ -\Lambda^E & 0 \end{pmatrix}
\end{equation*}
and making the associated gauge transformation (see \eqref{eq:changecomponents}). From \eqref{eq:4eq2}--\eqref{eq:4eq3}, we then see that, in the new gauge, $\Omega$ takes the form
\begin{equation*}
     \Omega = \begin{pmatrix} \omega & 0 \\ 0 & R^F \end{pmatrix},
\end{equation*}
where $R^F$ is the curvature of $\Delta^F$. Thus, for an appropriate choice of horizontal lift, the associated representation up to homotopy decomposes as the direct sum of a type $0$ representation up to homotopy on $\nu \oplus K[1]$ and a type $1$ representation up to homotopy on $F \oplus F[1]$.

For uniqueness, we use the classification of type $1$ and type $0$ $\VB$-groupoids in Proposition \ref{prop:type1} and Corollary \ref{cor:classtype0}. As a type $1$ $\VB$-groupoid, $\Gamma_1$ is determined up to isomorphism by the vector bundle $F$. As a type $0$ $\VB$-groupoid, $\Gamma_0$ is determined up to isomorphism by the representations $\Delta^\nu$ and $\Delta^K$ and the cohomology class $[\omega] \in H^2(G;\nu \to K)$. Both $\Delta^\nu$ and $\Delta^K$ are canonical, and it can be seen that the cohomology class $[\omega]$ is independent of the choices.
\end{proof}

The following result ties together the results in this section.
\begin{thm} \label{thm:classification}
Let $G \arrows M$ be a Lie groupoid, and let $E, C \to M$ be vector bundles. Regular $\VB$-groupoids over $G$ with side bundle $E$ and core $C$ are classified up to isomorphism by the following data:
\begin{itemize}
\item a regular bundle map $\boundary: C \to E$,
\item a representation $\Delta^K$ of $G$ on $K \defequal \ker \boundary$,
\item a representation $\Delta^{\nu}$ of $G$ on  $\nu \defequal \coker \boundary$, and
\item a cohomology class $[\omega] \in H^2(G; \nu \to K)$.
\end{itemize}
Therefore, regular representations up to homotopy of $G$ on $E \oplus C[1]$ are classified up to isomorphism by the same data.
\end{thm}

%====
\section{Example: the adjoint representation}\label{sec:adjoint}

Let $G \arrows M$ be a regular Lie groupoid, and let $A \to M$ be the Lie algebroid of $G$, with anchor map $\rho:A \to TM$. Consider the $\VB$-groupoid $TG \arrows TM$, which plays the role of the adjoint representation. By Theorem \ref{thm:classification}, there is a canonically-defined cohomology class $[\omega] \in H^2(G; \nu \to K)$, where $\nu \defequal \coker \rho$ and $K \defequal \ker \rho$. In this section, we give a geometric interpretation of the class $[\omega]$ in some nice cases.

\subsection{Restriction to orbits}
Before interpreting $[\omega]$, we require some notation and terminology. For $x \in M$, the \emph{orbit} through $x$ is $\orbit_x \defequal t(s^{-1}(x)) \subseteq M$. If the $s$-fibers of $G$ are connected, then $\orbit_x$ is the leaf through $x$ of the integrable distribution $F \defequal \im \rho \subseteq TM$, where $\rho: A \to TM$ is the anchor map for the Lie algebroid $A$ of $G$. We may restrict $G$ to a (immersed) Lie subgroupoid $G|_{\orbit_x} \arrows \orbit_x$, where $G|_{\orbit_x} \defequal s^{-1}(\orbit_x)$. On the other hand, it is clear that the disjoint union of $G|_{\orbit_x}$ (taken over all distinct $\orbit_x$) is isomorphic to $G$ as a set-theoretic groupoid. Intuitively, we should be able to recover the Lie groupoid structure of $G$ from ``gluing data'' describing how the different $G|_{\orbit_x}$ fit together.

\subsection{Bundles of Lie groups}

We first consider the case where $s=t$, so that $G$ is a bundle of Lie groups over $M$. In this case, each orbit consists of a single point $x \in M$, and in particular, the anchor map $\rho$ is zero, so $TG$ is a type $0$ $\VB$-groupoid. 

Let $h: G \times_M TM \to TG$ be a horizontal lift. Note that, in this situation, $h$ is equivalent to a connection on the fiber bundle $G \to M$. 

Since the representation of $G$ on $TM$ is trivial in this case, the formula for $\Omega$ in \eqref{eq:omega} simplifies slightly:
\begin{equation*}
     \Omega_{g_1, g_2} v = (h_{g_1 g_2}(v) - h_{g_1}(v) \cdot h_{g_2}(v)) \cdot \tilde{0}_{(g_1 g_2)^{-1}}
\end{equation*}
for $g_1, g_2 \in G_x$ and $v \in T_x M$. We can see that $\Omega$ measures the failure of parallel transport by $h$ to induce group homomorphisms of the fibers. Therefore, $G$ is locally trivializable as a Lie group bundle if and only if it is possible to choose $h$ such that $\Omega$ vanishes, and the cohomology class $[\Omega]$ can thus be viewed as an obstruction to such local trivializability.

\subsection{Regular Lie groupoids}
We now consider a more general case of a regular Lie groupoid $G \arrows M$. Let $x$ be a point in $M$. To avoid potential technical issues, we will assume that the orbit space is nice near $\orbit_x$; specifically, we assume that the quotient $M/\modsim$, where points in the same orbit are identified, has a smooth structure near $\orbit_x$ such that $\orbit_x$ is a regular value of the quotient map. As a result of this assumption, the quotient map induces an isomorphism between $\nu_y$ and $T_{\orbit_x} (M/\modsim)$ for all $y \in \orbit_x$, and therefore any $v \in \nu_x$ can be canonically extended to a section $\tilde{v}$ of $\nu|_{\orbit_x}$. 

Suppose that we have chosen splittings $TM \approx \nu \oplus F$ and $A \approx K \oplus F$, where $F = \im \rho$, and a horizontal lift $h: s^*(TM) \to TG$ such that the quasi-actions $\Delta^{TM}$ and $\Delta^{A}$, as well as the $2$-cochain $\Omega$, are block-diagonal with respect to the chosen splittings. The existence of such an $h$ is part of the proof of Lemma \ref{lemma:0and1}. Because of block-diagonality, we have that $\Delta^{TM}|_\nu = \Delta^\nu$, and one can see that $\Delta^\nu_g v = \tilde{v}_{t(g)}$ for any $v \in \nu_x$ and $g \in s^{-1}(x)$. We may then express the formula for $\omega = \Omega|_\nu$ as
\begin{equation*}
     \omega_{(g_1,g_2)} v \cdot \tilde{0}_{g_1 g_2} = h_{g_1 g_2}(\tilde{v}_x) - h_{g_1}(\tilde{v}_y) \cdot h_{g_2}(\tilde{v}_x)
\end{equation*}
for $v \in \nu_x$ and $(g_1, g_2) \in G^{(2)}$ such that $s(g_2) = x$. In other words, $\omega$ measures the infinitesimal failure of parallel transport by $h$ in the normal directions to give isomorphisms of the restricted groupoids $G|_{\orbit_x}$. 

We may give a more simple intepretation in the nicest case, where the orbit space is smooth and the quotient map $M \to M/\modsim$ is a submersion. In this case, we may think of $G$ as a ``bundle of transitive Lie groupoids'' over $M/\modsim$, where the fiber over $\orbit_x \in M/\modsim$ is the Lie groupoid $G|_{\orbit_x}$. It is possible to choose $h$ such that $\omega$ vanishes if and only if this bundle of transitive Lie groupoids is locally trivializable (so a ``transitive Lie groupoid bundle''). Therefore, the cohomology class $[\omega]$ can be viewed as an obstruction to such local trivializability.

It would be nice to find a simple interpretation of $[\omega]$ that does not require so many technical assumptions.

%==================================================================================

\section{The Fat category (groupoid)} \label{sec:fat}

In \cite{elw}, Evens, Lu, and Weinstein observed that the $1$-jet prolongation groupoid $J^1 G$ of a Lie groupoid $G \arrows M$, consisting of $1$-jets of bisections, carries natural representations on the Lie algebroid $A$ of $G$ and on $TM$. They noted that, although these representations do not pass to representations of $G$ on $A$ and $TM$, there is a sense in which they induce a ``representation up to homotopy'' on the complex $A \tolabel{\rho} TM$. However, the Arias Abad-Crainic notion of representation up to homotopy \cite{aba-cra:rephomgpd} which we have used in this paper differs from that of Evens-Lu-Weinstein. 

In this section, we describe a construction for $\VB$-groupoids that generalizes the $1$-jet prolongation groupoid, thus providing a conceptual link between the two notions of representation up to homotopy. Specifically, given a $\VB$-groupoid $\Gamma$, we construct a Lie groupoid $\fat{\Gamma} \to M$, with canonical representations on $C$ and $E$, such that $\fat{\Gamma} = J^1 G$ when $\Gamma = TG$. We then briefly describe how the formulas in \S\ref{sec:formulas} for the components of a representations up to homotopy can be derived directly from an object that is very closely related to $\fat{\Gamma}$.

\subsection{The fat category}
Let $\Gamma$ be a $\VB$-groupoid. The \emph{fat category} $\fatcat{\Gamma}$ consists of pairs $(g,H)$, where $g \in G$ and $H \subseteq \Gamma_g$ is a subspace that is complementary to the right-vertical subspace $V^R_g$ (see Definition \ref{dfn:rightvert}). There is an obvious projection map $\fatcat{\Gamma} \to G$, where the fiber over $g \in G$ is an affine space modeled on $\Hom(E_{s(g)},C_{t(g)})$. Thus, $\fatcat{\Gamma}$ is a smooth manifold.

As the name suggests, $\fatcat{\Gamma}$ has the structure of a Lie category, where the source and target maps factor through the projection $\fatcat{\Gamma} \to G$ and the multiplication is given by $(g_1, H_1) \cdot (g_2, H_2) = (g_1 g_2, H_1 H_2)$, where 
\begin{equation*}
 H_1 H_2 = \{ v_1 \cdot v_2 \suchthat v_i \in H_i \mbox{ and } \tilde{s}(v_1) = \tilde{t}(v_2)\}.    
\end{equation*}
If $(g,H)$ is invertible, then its inverse is simply $(g^{-1}, H^{-1})$. However, it is possible for $(g,H)$ to be noninvertible; this occurs exactly when $H$ fails to be complementary to $V^L_g$.

\subsection{The fat groupoid}
The \emph{fat groupoid} $\fat{\Gamma}$ consists of all invertible elements of $\fatcat{\Gamma}$, i.e.\ pairs $(g,H)$, where $g \in G$ and $H \subseteq \Gamma_g$ is a subspace complementary to both $V^R_g$ and $V^L_g$. The elements of $\fat{\Gamma}$ form an open subset of $\fatcat{\Gamma}$, so $\fat{\Gamma}$ naturally inherits a smooth Lie groupoid structure.

In the case where $\Gamma = TG$, the fat groupoid $\fat{\Gamma}$ consists of pairs $(g, H)$, where $g \in G$ and $H \subseteq T_g G$ is a subspace complementary to both the source-fiber and the target-fiber; in other words, $H$ is the $1$-jet of a bisection of $G$. Thus, in this case, the fat groupoid $\fat(TG)$ is the $1$-jet prolongation groupoid $J^1 G$.

\subsection{Representations of the fat category (groupoid)}
The fat category has canonical Lie category representations $\psi^C$ and $\psi^E$ on the vector bundles $C$ and $E$, respectively, defined as follows. For $e \in E_{s(g)}$ and $c \in C_{s(g)}$, 
\begin{align*}
\psi^E_{(g,H)} e &= \tilde{t}(v),\\
\psi^C_{(g,H)} c &= w \cdot c \cdot \tilde{0}_{g^{-1}}
\end{align*}
where $v$ is the unique vector in $H$ such that $\tilde{s}(v) = e$, and $w$ is the unique vector in $H$ such that $\tilde{s}(w) = \tilde{t}(c)$. The representations $\psi^C$ and $\psi^E$ restrict in the obvious way to produce Lie groupoid representations of the fat groupoid, which we will also denote as $\psi^C$ and $\psi^E$.

Recall that the core-anchor $\boundary: C \to E$ is given by $\boundary(c) = \tilde{t}(c)$. It follows that the representations $\psi^C$ and $\psi^E$ are related by the core-anchor: $\boundary \psi^C = \psi^E \boundary$.

\subsection{Sections and representations up to homotopy}
We would like to pass the canonical representations $\psi^C$ and $\psi^E$ of $\fat{\Gamma}$ to $G$. The obvious way to do so would be to choose a section of the projection map $\fat{\Gamma} \to G$ and then use the section to pull $\psi^C$ and $\psi^E$ back to $G$. However, such a section does not always exist globally. On the other hand, global sections do always exist for the projection $\fatcat{\Gamma} \to G$; indeed, such a section is equivalent to a section of the short exact sequence \eqref{eq:ses}. 

We may impose the additional requirement that a unit $1_x$ lift to $(1_x, \tilde{1}(E_x))$. Sections of $\fatcat{\Gamma} \to G$ satisfying this requirement are equivalent to (right-)horizontal lifts of $\Gamma$. Specifically, given a horizontal lift $h: s^* E \to \Gamma$, the map $g \mapsto \hat{g} \defequal (g, h_g(E_{s(g)}))$ is a section of $\fatcat{\Gamma} \to G$. If we use this section to pull the representations $\psi^C$ and $\psi^E$ back to $G$, we immediately recover the formulas \eqref{eq:coreconnection}-\eqref{eq:sideconnection}. 

In general, we can't expect the lift $g \mapsto \hat{g}$ to respect multiplication, which is why $\Delta^C$ and $\Delta^E$ are only quasi-actions and not representations. The failure of the lift to respect multiplication is measured by $\widehat{g_1 g_2} - \hat{g_1} \cdot \hat{g_2}$, which can be identified with $\Omega_{g_1,g_2}$ as given by \eqref{eq:omega}.

%====
\appendix
\section{Derivation of the representation up to homotopy formulas}
\label{appendix:derivation}

\subsection{Horizontal lifts and dual pairings}

Let $\Gamma$ be a $\VB$-groupoid. Throughout this section, we will assume that $\Gamma$ has a fixed right-horizontal lift $h : s^*E \to \Gamma$. By \eqref{eq:VH}, any $\gamma \in \Gamma_g$ may be uniquely written as
\begin{equation*}
 \gamma = \gamma^V \cdot \tilde{0}_g + h_g(\gamma^H),
\end{equation*}
where $\gamma^H = \tilde{s}(\gamma) \in E_{s(g)}$ and $\gamma^V \in C_{t(g)}$. We refer to $\gamma^V$ and $\gamma^H$, respectively, as the vertical and horizontal parts of $\gamma$. Similarly, any $\xi \in \Gamma^*_g$ may be uniquely written as
\begin{equation*}
 \xi = \du{0}_g \cdot \xi^V + \eta_g(\xi^H),
\end{equation*}
where $\xi^H = \du{t}(\xi) \in C^*_{t(g)}$ and $\xi^V \in E^*_{s(g)}$, with $E^*$ identified with the \emph{left}-core of $\Gamma^*$. Here, $\eta : t^*C^* \to \Gamma^*$ is the left-horizontal lift given by the equation
\begin{equation}\label{eq:dualpairc}
 \pair{\eta_g(\nu)}{\gamma} = \pair{\nu}{\gamma^V}
\end{equation}
for $\nu \in C^*_{t(g)}$ and $\gamma \in \Gamma_g$. Equation \eqref{eq:dualpairc} defines a one-to-one correspondence between right-horizontal lifts on $\Gamma$ and left-horizontal lifts on $\Gamma^*$. We also have the following equation, which is a consequence of the fact that the natural inclusion $s^*E^* \to \Gamma^*$, $(g,\tau) \mapsto \du{0}_g \cdot \tau$, is the dual of the projection $\Gamma \to s^*E$, $\gamma \mapsto (\tilde{q}(\gamma),\gamma^H)$:
\begin{equation}\label{eq:dualpaire}
 \pair{\du{0}_g \cdot \tau}{\gamma} = \pair{\tau}{\gamma^H}.
\end{equation}
Together, \eqref{eq:dualpairc} and \eqref{eq:dualpaire} allow us to simply express the pairing of $\gamma$ and $\xi$ as
\begin{equation}\label{eq:dualpair}
 \pair{\xi}{\gamma} =  \pair{\xi^V}{\gamma^H} + \pair{\xi^H}{\gamma^V}.
\end{equation}
In particular, horizontal elements of $\Gamma^*$ annihilate horizontal elements of $\Gamma$, and vertical elements of $\Gamma^*$ annihilate vertical elements of $\Gamma$. 

We conclude the section with a lemma that will be useful in \S\ref{sec:components}.

\begin{lemma}\label{lemma:duspair}
For any $\xi \in \Gamma^*_g$ and $c \in C_{s(g)}$,
\begin{equation*}
 \pair{\du{s}(\xi)}{c} = \pair{\xi^V}{-\tilde{t}(c)} + \pair{\xi^H}{h_g(\tilde{t}(c))\cdot c \cdot \tilde{0}_{g^{-1}}}.
\end{equation*}
\end{lemma}
\begin{proof}
From the definition of $\du{s}$ in \eqref{eq:sdu}, we have
\begin{equation*}
 \pair{\du{s}(\xi)}{c} =  \pair{\xi}{\tilde{0}_g \cdot (-c^{-1})},
\end{equation*}
which may be decomposed via \eqref{eq:dualpair}. The horizontal part of $\tilde{0}_g \cdot (-c^{-1})$ is $\tilde{s}(\tilde{0}_g \cdot (-c^{-1})) = \tilde{s}(-c^{-1}) = -\tilde{t}(c)$. To obtain the vertical part, we subtract the horizontal lift of the horizontal part:
\begin{equation*}
\begin{split}
 (\tilde{0}_g \cdot (-c^{-1}))^V \cdot \tilde{0}_g &= \tilde{0}_g \cdot (-c^{-1}) + h_g(\tilde{t}(c))\\
&=  \tilde{0}_g \cdot (-c^{-1}) + h_g(\tilde{t}(c)) \cdot \tilde{1}_{\tilde{t}(c)}\\
&= (\tilde{0}_g + h_g\tilde{t}(c))\cdot(-c^{-1} + \tilde{1}_{\tilde{t}(c)})\\
&= h_g\tilde{t}(c) \cdot c.
\end{split}
\end{equation*}
In the last line, we have used \eqref{eq:rightleft}. Thus the vertical part of $\tilde{0}_g \cdot (-c^{-1})$ is $h_g(\tilde{t}(c))\cdot c \cdot \tilde{0}_{g^{-1}}$, and the result follows from \eqref{eq:dualpair}.
\end{proof}

\subsection{Formulas for representation up to homotopy components}\label{sec:components}
We wish to show that the formulas \eqref{eq:boundary}--\eqref{eq:omega} for the four components $\boundary$, $\Delta^C$, $\Delta^E$, and $\Omega$ agree with the representation up to homotopy $\totaldiff_h$ in Corollary \ref{cor:superreph}. We will do this by applying $\totaldiff_h$ to $0$-cochains $\alpha \in \Gamma(C) = C^0(G;C)$ and $\varepsilon \in \Gamma(E) = C^0(G;E)$, and showing that the resulting cochains decompose as
\begin{align}
 \totaldiff_h \alpha &= D^C \alpha + \boundary \alpha, \label{eq:totalalpha}\\
\totaldiff_h \varepsilon &= \Omega \varepsilon + D^E \varepsilon. \label{eq:totalepsilon}
\end{align}

First, we consider $\totaldiff_h \alpha$. By alternatively viewing $\alpha$ as a section of $C$ and as a linear function on $C^*$, we may write $\vbiso_h^{-1} \alpha = \alpha$, so $\totaldiff_h \alpha = -\vbiso_h (-\du{\delta} \alpha)$. By the definition of the differential $\du{\delta}$, 
\begin{equation*}
-\du{\delta}\alpha (\xi) = \pair{\du{t}(\xi)}{\alpha_{t(g)}} - \pair{\du{s}(\xi)}{\alpha_{s(g)}}
\end{equation*}
for $\xi \in \Gamma^*_g$. Using \eqref{eq:tdu}, \eqref{eq:dualpair}, and Lemma \ref{lemma:duspair}, we can rewrite this as
\begin{equation*}
\begin{split}
&\pair{\xi^H}{\alpha_{t(g)} - h_g(\tilde{t}(\alpha_{s(g)})) \cdot \alpha_{s(g)} \cdot \tilde{0}_{g^{-1}}} + \pair{\xi^V}{\tilde{t}(\alpha_{s(g)})} \\
&= \pair{\xi^H}{\alpha_{t(g)} - \Delta^C_g \alpha_{s(g)}} + \pair{\xi^V}{\boundary \alpha_{s(g)}}.
\end{split}
\end{equation*}
Using \eqref{eqn:difftorepodd} and comparing \eqref{eq:hatce} with \eqref{eq:dualpair}, we conclude that \eqref{eq:totalalpha} does indeed hold.

Next, we consider $\totaldiff_h \varepsilon$. Equations \eqref{eq:vbhat} and \eqref{eq:hatce} imply that $\vbiso_h^{-1}\varepsilon \in C^1_\VB(\Gamma)$ is given by
\begin{equation*}
 \vbiso_h^{-1}\varepsilon(\xi) = \pair{\xi}{h_g (\varepsilon_{s(g)})}
\end{equation*}
for $\xi \in \Gamma^*_g$. By the definition of $\du{\delta}$, 
\begin{equation*}
\begin{split}
-\du{\delta}( \vbiso_h^{-1} \varepsilon) (\xi_1, \xi_2) &= -\vbiso_h^{-1} \varepsilon(\xi_2) + \vbiso_h^{-1} \varepsilon(\xi_1 \cdot \xi_2) - \vbiso_h^{-1} \varepsilon(\xi_1) \\
&= -\pair{\xi_2}{h_{g_2}(\varepsilon_{s(g_2)})} + \pair{\xi_1 \cdot \xi_2}{h_{g_1 g_2}(\varepsilon_{s(g_2)})} - \pair{\xi_1}{h_{g_1}(\varepsilon_{s(g_1)})}
\end{split}
\end{equation*}
for $(\xi_1,\xi_2) \in (\Gamma^*)^{(2)}_{(g_1,g_2)}$. Using \eqref{eq:omega} and \eqref{eq:defmult}, we can rewrite the middle term as
\begin{equation*}
\begin{split}
&\pair{\xi_1 \cdot \xi_2}{\Omega_{g_1,g_2}\varepsilon_{s(g_2)} \cdot \tilde{0}_{g_1g_2} + h_{g_1}(\Delta^E_{g_2} \varepsilon_{s(g_2)}) \cdot h_{g_2}(\varepsilon_{s(g_2)})} \\
&= \pair{\xi_1}{\Omega_{g_1,g_2}\varepsilon_{s(g_2)} \cdot \tilde{0}_{g_1} + h_{g_1}(\Delta^E_{g_2} \varepsilon_{s(g_2)})} + \pair{\xi_2}{h_{g_2}(\varepsilon_{s(g_2)})}.
\end{split}
\end{equation*}
Substituting this into the previous equation and using \eqref{eq:dualpair}, we get
\begin{equation*}
\begin{split}
-\du{\delta}( \vbiso_h^{-1} \varepsilon) (\xi_1, \xi_2) &= \pair{\xi_1}{\Omega_{g_1,g_2}\varepsilon_{s(g_2)} \cdot \tilde{0}_{g_1} + h_{g_1}(\Delta^E_{g_2} \varepsilon_{s(g_2)} - \varepsilon_{s(g_1)})} \\
&= \pair{\xi_1^H}{\Omega_{g_1,g_2}\varepsilon_{s(g_2)}} + \pair{\xi_1^V}{ \Delta^E_{g_2} \varepsilon_{s(g_2)} - \varepsilon_{s(g_1)}} .
\end{split}
\end{equation*}
Using \eqref{eqn:difftorepeven} and comparing \eqref{eq:hatce} with \eqref{eq:dualpair}, we conclude that \eqref{eq:totalepsilon} does indeed hold.

\bibliographystyle{abbrv}
\bibliography{bibio}
\end{document}